\documentclass{article}
\usepackage[utf8x]{inputenc}
\usepackage{graphicx}
\usepackage{wrapfig}
\usepackage[pdftex,bookmarks]{hyperref}
\hypersetup{
pdftitle={boolean},
pdfauthor={José M. Zapata}
}

\usepackage{fancyhdr}
\usepackage{stmaryrd}

\usepackage{mathrsfs} 
\usepackage{relsize}

\usepackage[english]{babel}
\usepackage{amsmath}
\usepackage{amsthm}
\usepackage{amsfonts}

\usepackage{enumitem}

\usepackage{amscd}

\usepackage{epigraph}

\usepackage{amssymb}

\newcommand{\Q}{\mathbb{Q}}
\newcommand{\N}{\mathbb{N}}
\newcommand{\R}{\mathbb{R}}

\newcommand{\PP}{\mathbb{P}}
\newcommand{\VA}{V^{(\A)}}
\newcommand{\VF}{V^{(\mathcal{F})}}

\newcommand{\F}{\mathcal{F}}

\newcommand{\A}{\mathcal{A}}

\DeclareMathOperator*{\esssup}{ess.sup\,}

\DeclareMathOperator{\dom}{dom}
\DeclareMathOperator{\Ord}{Ord}

\newtheorem{defn}{Definition}[section]
\newtheorem{rem}{Remark}[section]
\newtheorem{prop}{Proposition}[section]

\newtheorem{thm}{Theorem}[section]

\providecommand{\keywords}[1]{\textbf{\textit{Keywords:}} #1}

\begin{document}
\title{Boolean-valued models as a foundation for locally $L^0$-convex analysis and Conditional set theory\thanks{The authors would like to thank an anonymous referee for a careful review of the manuscript and valuable comments, and for pointing out a stem of excellent related works.}}
\author{A. Avilés \thanks{Universidad de Murcia, Dpto. Matemáticas, 30100 Espinardo, Murcia, Spain, e-mail: avileslo@um.es. }, J.M. Zapata \thanks{Universidad de Murcia, Dpto. Matemáticas, 30100 Espinardo, Murcia, Spain, e-mail: jmzg1@um.es. } \thanks{Second author was partially supported by the grant MINECO MTM2014-57838-C2-1-P}}
\date{}
\maketitle

\begin{abstract}
Locally $L^0$-convex modules were introduced in [D. Filipovic, M. Kupper, N. Vogelpoth. Separation and duality in locally $L^0$-convex modules. J. Funct. Anal. 256(12), 3996-4029 (2009)] as the analytic basis for the study of  multi-period mathematical finance. Later, the algebra of conditional sets was introduced in [S. Drapeau, A. Jamneshan, M. Karliczek, M. Kupper. The algebra of conditional sets and the concepts of conditional topology and compactness. J. Math. Anal. Appl. 437(1), 561-589 (2016)]. 
By means of Boolean-valued models and its transfer principle we show that any known result on locally convex spaces has a transcription in the frame of locally $L^0$-convex modules which is also true, and that the formulation in conditional set theory of any theorem of classical set theory is also a theorem.   
We propose Boolean-valued analysis as an analytic framework for the study of multi-period problems in mathematical finance. 




\end{abstract}

\section*{Introduction}

Boolean-valued models are a tool in mathematical logic that was developed as a way to formalize the method of forcing that Paul Cohen created to solve the first problem in the famous Hilbert's list: it is impossible neither to prove nor to disprove that every infinite set of reals can be bijected either with the natural numbers of with the whole real line  \cite{cohen1966set}. The theory was first formulated by Scott \cite{scott1967proof} based on some ideas of Solovay, while Vop\v{e}nka created independently a similar theory. In this paper we shall see that Boolean-valued models provide a natural framework for certain problems in financial mathematics which involve a multi-period setting, such as representation of dynamic and conditional risk measures and stochastic optimal control. 
Several recent developments, like the study of \emph{locally $L^0$-convex modules}  and the \emph{algebra of conditional sets} are covered by this theory. 
The advantage is not only a unified approach for several scattered results in the literature; the important point is that we get at our disposal all the powerful tools of a well developed deep mathematical theory. 
In particular, the so-called \emph{transfer principle} claims that any known result of set theory has a transcription in the Boolean-valued setting, which is also true.  

In order to provide an analytical basis to problems of mathematical finance in a multi-period set-up with a dynamic flow of information, Filipovic et al.~\cite{kupper03} considered $L^0:=L^0(\Omega,\Sigma,\PP)$, the ordered lattice ring of equivalence classes modulo almost sure equality of $\Sigma$-measurable random variables, where $(\Omega,\Sigma,\PP)$ is a probability space that models the market information that is available at some future time. 
They introduced the topology of almost sure dominance on $L^0$ and the notion of locally $L^0$-convex module, and succeeded in  developing a randomized version of classical convex analysis. 
We will show that a locally $L^0$-convex module can be embedded into a Boolean-valued universe and will systematically study the  meaning of several related objects within this framework. 
Thus, we will show that not only the main results of \cite{kupper03} are consequence of this connection but that any known  result of locally convex analysis has a modular transcription which also holds as a consequence of the transfer principle of Boolean-valued models. 
For instance, we provide randomized versions of celebrated theorems such as the Brouwer's fixed point and the James' compactness theorem. 

 
Conditional set theory was introduced in \cite{DJKK13}. The authors need to give a definition of the conditional version of each mathematical concept they want to use (conditional real number, conditional topological space, etc.)  and also formulate and prove the conditional version of known results. 
All this is automatic within the context that we propose: the conditional version of any known result is automatically true. 
Nevertheless, we should mention that this connection exhibits that conditional set theory provides a practical tool to manage objects from Boolean-valued models and gives intuition to anyone who is not familiarized with the formalisms of Boolean-valued  models.

These developments have been applied to mathematical finance. 
For instance, we find applications to representation of conditional risk measures (see eg~\cite{bielecki2016dynamic,kupper11,frittelli2011dual}), to equilibrium theory (see \cite{BH14,kupper08}), to optimal  stochastic control (see \cite{JKZ2017control}) and to financial preferences (see \cite{DJ13}).  
Also, as commented, there is a significant number of related works. 
For instance, modules endowed with $L^0$-valued norms have been applied to the study of ultrapowers of Lebesgue-Bochner spaces by Haydon et al.~\cite{haydon1991randomly}. 
 Guo together with other co-authors have widely studied generalizations of functional analysis results in $L^0$-modules endowed with the topology of stochastic convergence with respect to a family of $L^0$-valued seminorms, also called the $(\epsilon,\lambda)$-topology (see eg~\cite{guo08,guo2013homo,guo2009random} and references therein). 
In this regard, a study of the relations between the $(\epsilon,\lambda)$-topology and the locally $L^0$-convex topology induced by $L^0$-valued seminorms can be found in \cite{guo10}. 
Eisele and Taieb \cite{eisele13} extended some functional analysis results to modules over the ring $L^\infty$. 
A randomized version of finite-dimensional analysis in $(L^0)^d$ together with many results are provided in \cite{Cheridito2012} and also  a version of the Brouwer fixed point theorem in this context is  established in \cite{DKKM13}. 
Other results and counter-examples on locally $L^0$-convex modules can be found in \cite{zapata2017characterization,zapataportfolio}. 
Further studies of dual pairs and weak topologies in the context of conditional sets are provided in \cite{OZ2017stabil,zapataweak}. 
  
In addition, we should highlight that the Boolean-valued model approach shows that the study of locally $L^0$-convex modules naturally fits the framework of the well-developed theory of lattice-normed spaces (i.e. norms that take values in a vector lattice) and dominated operators, originated in the 1930s by L.V. Kantorovich (see~\cite{Kantorovich}), field that has been widely researched and fruitfully exploited by A. G. Kusraev and S. S. Kutateladze.  
For a thorough account we refer the reader to eg~\cite{kusraev2000dominated,kusraev2012boolean} and their extensive list of references.

The paper is organized as follows: In the first section we give a short introduction to Boolean-valued models, provide some intuition and recall the basic elements and principles of the Boolean-valued machinery. In Section 2, we explain a precise connection between the framework of locally $L^0$-convex modules and Boolean-valued locally convex analysis; give a list of basic elements of locally $L^0$-convex analysis and explain their meanings within Boolean-valued locally convex analysis; and as example of application we derive the main theorems of \cite{kupper03} and modular versions of James' compactness theorem and Brouwer's fixed point theorem. Finally, Section 3 is devoted to provide a precise connection between conditional set theory and Boolean-valued models.  \\

\section{Foundations of Boolean-valued models}

Let us try to give an intuitive idea of what Boolean-valued analysis is and how it can fit in mathematical finance. We would like to talk about what will happen in a particular moment in the future. This future is uncertain, it is influenced by events that we do not know yet. These possible events that might influence the future will be coded by a complete Booolean algebra $\A=(\A,\vee,\wedge,{}^c,0,1)$. The simplest case that we can think of is that the future that we are interested in is completely determined by the result of flipping a coin. In that case, the algebra of events is $\mathcal{A}_0= \{a,a^c,0,1\}$ where $a$ is the event ``we get head'', and its negation $a^c$ is ``we get tail''. In the algebra of events we also have all the events that we can formulate combining others, and so $a\vee a^c = 1$ is the event ``we get either head or tail'', which is just the true event, and $a\wedge a^c = 0$ is the event ``we get both head and tail'' which is just the false event. Another example is that our future depends on a randomly chosen (say with Gaussian probability) real number. In that case, the algebra of events would be the measure algebra: measurable subsets of $\mathbb{R}$ modulo null sets. In that case, for example, the class of $[1,2]$ is interpreted as the event ``the random number happened to fall inside the interval $[1,2]$''.\\

So let us fix the algebra $\mathcal{A}$ of all the events that we can talk about and influence the future. The next element of our theory are the \emph{names}. The names are the nouns of the language with which we talk about objects in the future despite the uncertainties. In the flipping coin example, suppose that I have five dollars and I bet two dollars that the outcome will be tail. Then I can consider the name $\dot{x}$ that represents the amount of money that I will have in the future. The actual value of $\dot{x}$ is unknown, it could be 3 or 7 depending on the coin. In the very simple flipping coin case, a name can be identified with a pair $(r,s)$ of mathematical objects, one for head and one for tail. In the random real case, names will look more complicated but the idea is similar. Examples of names would be $\dot{y}$ that would take value 1 if the random real is positive or -1 if it is negative, and also $\dot{z}$ the name for the random real itself. 
A special kind of names are those which do not really depend on the unknown events, and those are represented with a $\vee$ symbol above. For example, $\check{5}$ is a name which represents the number 5, no matter what the coin did or what the random real actually happens to be.\\

Once we understand the idea of a name, the next step is formulating statements about names and deciding what are the truth values of such statements. Playing with the names given above in the flipping coin case, it make sense to make the following statements: P1. $\dot{x}$ is a positive real number, P2. $\dot{x} = \check{7}$, P3. $\dot{x} < \check{4}$, P4. $\dot{x} = \dot{x}^2$. While $P1$ is clearly true and $P4$ is clearly false, for P2 and P3 we may say that \emph{it depends on what the coin will do}. In Boolean-valued analysis, these statements are not assigned a binary truth value of true or false. The truth valued of a sentence $P$, denoted by $\llbracket P\rrbracket$ is an element of the Boolean algebra $\mathcal{A}$ that corresponds to the event that describes when this sentence is true. Thus $\llbracket P3\rrbracket = a^c$ because I have 8 dollars if and only if the flipping will give a tail, and similarly $\llbracket P4\rrbracket = a$. In the random case, for instance, the truth value of the sentence $\check{2}\dot{z}<\check{4}$ is exactly the representative of the interval $(-\infty,2)$ of the measure algebra, while $\llbracket \dot{z}>\dot{y}\rrbracket$ is the representative of $(-1,0)\cup (1,+\infty)$. In all these examples, we are using names for real numbers, but the idea is more general, we can have names for functions, sets, Banach spaces or any mathematical object we want, and state any kind of properties we wish in formal mathematical language. In the framework of set theory, any mathematical object can be considered as a set and any mathematical statement can be re-stated in terms of the belonging relation $\in$ between sets.\\

The precise formulation of Boolean-valued analysis requires some familiarity with the basics of set theory and logic, and in particular with first order logic, ordinals and transfinite induction. However, if one understands the key ideas and principles, it is possible to work with Boolean-valued models avoiding the underlying machinery that can be conveniently hidden in a black box. For a detailed description we can refer the reader to \cite{bell2005set}, \cite[Chapter 14]{jech2013set}, or \cite[Chapter 2]{kusraev2012boolean}. We make now a quick review.\\

Let us consider a universe of sets $V$ satisfying the axioms of the Zermelo-Fraenkel set theory with the axiom of choice (ZFC), and  a first-order language $\mathcal{L}$ which allows the formulation of statements about the elements of $V$. In the universe $V$ we have all possible mathematical objects (real numbers, topological spaces, etc.) that we can talk about in a context of total certainty. The language $\mathcal{L}$ consists of the elements of $V$ plus a finite list of symbols for logic symbols ($\forall$, $\wedge$, $\neg$ and parenthesis), variables (with the symbol $x$ we can express any variables we need as $x,xx,xxx,\ldots$) and the verbs  $=$ and $\in$.  Though we usually use a much richer language by introducing more and more intricate  definitions, in the end any usual mathematical statement can be written using only those mentioned. The elements of the universe $V$ are classified into a transfinite hierarchy: $V_0\subset V_1\subset V_2 \subset \cdots V_\omega \subset V_{\omega+1}\subset \cdots$, where $V_0 = \emptyset$, $V_{\alpha+1} = \mathcal{P}(V_\alpha)$ is the family of all sets whose elements come from $V_\alpha$, and $V_\beta = \bigcup_{\alpha<\beta}V_\alpha$ for limit ordinal $\beta$.\\

Now consider the complete Boolean algebra of events $\A=(\A,\vee,\wedge,{}^c,0,1)$  which is an element of $V$.
For given $a,b\in\A$, we will write $a\leq b$ whenever $a\wedge b= a$. For a family $\{a_i\}_{i\in I}$ in $\A$, we denote its supremum by $\bigvee_{i\in I} a_i$ and its infimum by $\bigwedge_{i\in I} a_i$. A family $\{a_i\}_{i\in I}$ in $\A$ is said to be a partition of $a$ if $\bigvee_{i\in I} a_i=a$ and $a_i\wedge a_j=0$ for all $i\neq j$, $i,j\in I$ (notice that $I$ could be infinite in this definition). For given $a\in\A$ we denote by $p(a)$ the set of all partitions of $a$.\\

 Given this complete Boolean algebra $\A$, one constructs now $\VA$, the \emph{Boolean-valued model} of $\A$, whose elements are the \emph{names} that we mentioned earlier, that we interpret as nouns with which we talk about the future.
 We proceed by induction over the class $\Ord$ of ordinals of the universe $V$. We start by defining $\VA_0:=\emptyset$; 
 if $\alpha+1$ is the successor of $\alpha$, we define
 $$\VA_{\alpha+1}:=\left\{ x \colon x\textnormal{ is an $\A$-valued function with }\dom(x)\subset\VA_\alpha  \right\}.$$
 The idea is that for $y\in dom(x)$, $y$ will become an element of $x$ in the future if $x(y)$ happens.
 If  $\alpha$ is a limit ordinal  $V_\alpha^{(\A)}:=\underset{\xi<\alpha}\bigcup V_{\xi}^{(\A)}$. Finally, let $V^{(\A)}:=\underset{\alpha\in Ord}\bigcup V_{\alpha}^{(\A)}.$
 
 Given an element $x$ in $\VA$ we define its \emph{rank} as the least ordinal $\alpha$ such that $x$ is in $\VA_{\alpha+1}$.
 
 We consider a first-order language which allows to produce statements about $\VA$.
 Namely, let $\mathcal{L}^{(\A)}$ be the first-order language which is the extension of $\mathcal{L}$ by adding names for each element of $\VA$. Suppose that $\varphi$ is any formula of the language $\mathcal{L}^{(\A)}$, its \emph{Boolean truth value} $\llbracket\varphi\rrbracket$  is defined by induction in the length of $\varphi$. If one got the right intuition, all the formulas that follow should look natural. We start by defining the Boolean truth value of the \emph{atomic formulas} $x\in y$ and $x=y$ for $x$ and $y$ in $\VA$. Namely, proceeding by transfinite recursion we define
 \[
 \llbracket x\in y\rrbracket=\underset{t\in\dom(y)} \bigvee y(t)\wedge\llbracket t=x\rrbracket,
 \]
 \[
 \llbracket x=y\rrbracket=\underset{t\in\dom(x)}\bigwedge \left(x(t)\Rightarrow \llbracket t\in y\rrbracket\right) \wedge \underset{t\in\dom(y)}\bigwedge \left(y(t)\Rightarrow \llbracket t\in x\rrbracket\right),
 \]

where, for $a,b\in\A$, we denote $a\Rightarrow b:=a^c\vee b$.
 For non-atomic formulas we have 
 \[
 \llbracket \exists x\varphi(x)\rrbracket:=\underset{u\in \VA}\bigvee \llbracket \varphi(u)\rrbracket\quad\textnormal{ and }\quad \llbracket \forall x\varphi(x)\rrbracket:=\underset{u\in \VA}\bigwedge \llbracket \varphi(u)\rrbracket;
 \]
\[
\llbracket \varphi \wedge \psi\rrbracket:=\llbracket \varphi\rrbracket \wedge \llbracket\psi\rrbracket \quad\textnormal{ and }\quad \llbracket \neg\varphi\rrbracket:=\llbracket\varphi\rrbracket^c.
\]
It is well-known that every theorem of ZFC is true in $\VA$ with the Boolean truth value:
\begin{thm}(Transfer Principle)
If $\varphi$ is a theorem of ZFC, then $\llbracket\varphi\rrbracket=1$.
\end{thm}  

Also, it will be important to keep in mind the following results, which will allow to manipulate $\VA$ and are well-known within Boolean-valued models theory:

\begin{thm}(Maximum Principle)
\label{thm: Mixing}
Let $\varphi(x)$ be a formula with one free variable $x$. Then there exists an element $u$ of $\VA$ such that $\llbracket\varphi(u)\rrbracket=\llbracket \exists x\varphi(x)\rrbracket$. 

\end{thm}

\begin{thm}(Mixing Principle)
\label{thm: Mixing}
Let $\{a_i\}\in p(1)$ and let $\{x_i\}$ be a family in $\VA$. Then there exists an element $x$ in $\VA$ such that $\llbracket x=x_i\rrbracket\geq a_i$ for all $i$. Moreover, if $y$ is another element of $\VA$ which satisfies the same, then $\llbracket x=y\rrbracket=1$.
\end{thm}

Let us say that two names $x,y$ are equivalent, and write $x\sim y$, when $\llbracket x=y \rrbracket = 1$. The truth value of a formula is not affected when we change a name by an equivalent one. Given a set $x$ in $V$ we define its canonical name $\check{x}$ in $\VA$. Namely, we put $\check{\emptyset}:=\emptyset$ and for $x$ in $\VA$ we define $\check{x}:\mathcal{D}\rightarrow \A$, where $\mathcal{D}:=\left\{\check{y}\colon y\in x \right\}$ and $\check{x}(\check{y}):=1$ for $y\in x$. 
It is not difficult to show that $\check{x}$ is an element of $\VA$. If $x,y,f\in \VA$ and we say, for instance, that $f$ is a name for a function $f:x\rightarrow y$, this means that $\llbracket\textnormal{``} f \text{ is a function from }x\text{ to }y\textnormal{''}\rrbracket = 1$. 
The transfer and maximum principles provide us with names $\mathbb{N}^{(\mathcal{A})}$ and $\mathbb{R}^{(\mathcal{A})}$ for the sets of natural numbers and real numbers, respectively. 
This means: $\llbracket\textnormal{``} \mathbb{N}^{(\mathcal{A})} \text{ is the set of natural numbers''}\rrbracket = 1$. 


Let $\overline{V}^{(\A)}$ be the subclass of $\VA$ defined by choosing a representative of the least rank in each class of the equivalence relation $\{(x,y)  \colon \llbracket x=y\rrbracket=1 \}$.\footnote{The construction can be done by transfinite induction. We choose a representative of each class $\{(x,y)\colon x,y\in \VA_{\alpha+1}\colon \llbracket x=y\rrbracket=1,\: \llbracket x=z\rrbracket<1\textnormal{ for al }z\in\overline{V}^{(\A)}_\alpha\}$ and define $\overline{V}_{\alpha+1}^{(\A)}$ the set of all theses representatives. For a limit ordinal $\alpha$ we put $\VA_\alpha:=\bigcup_{\xi<\alpha}\overline{V}^{(\A)}_\xi$. The class $\overline{V}^{(\A)}$ is frequently defined in literature and is called the \emph{separated universe}, see eg~\cite{kusraev2012boolean,vladimirov2013boolean}.} 
Given a name $x$ with $\llbracket x\neq\emptyset\rrbracket=1$ we define its \emph{descent} by  
\[
x\downarrow = \{ y\in\overline{V}^{(\A)} \colon \llbracket y\in x\rrbracket=1\}.
\]
Notice that, if $x\in V^{(\mathcal{A})}_\alpha$, then any element of the class $x\downarrow$ is also in $V^{(\mathcal{A})}_\alpha$. 
Therefore, we have that $x\downarrow$ is a set in $V$.  


The following result will be useful later:

\begin{thm}
\label{thm: extensional}
Let $x,y$ be elements of $\VA$ with $\llbracket (x\neq \emptyset)\wedge (y\neq\emptyset)\rrbracket=1$, let $f:x\downarrow\rightarrow y\downarrow$ be a function such that 
$$\llbracket u=v\rrbracket\leq \llbracket f(u)=f(v)\rrbracket\quad\textnormal{ for all }u,v\in x\downarrow.$$
Then there exists $g$ in $\VA$, which is a name for a function between $x$ and $y$, 
such that $\llbracket f(u)=g(u)\rrbracket=1$ for all $u\in x\downarrow$.  
\end{thm}

\section{A precise connection between locally $L^0$-convex analysis and Boolean-valued locally convex analysis}

Let $(\Omega,\Sigma,\PP)$ be a probability space of the universe $V$ and let $L^0$ denote the set of $\Sigma$-measurable random variables, which are identified whenever their difference is $\PP$-negligible. 
 We denote by $\F$ the measure algebra, which is defined by identifying events whose symmetric difference has probability $0$. 
 Then, $\F$ has structure of complete Boolean algebra which satisfies the \emph{countable chain condition}, that is, all partitions  are at most countable. 
 Since $\F$ is a complete Boolean algebra, one can consider the corresponding boolean-valued model $\VF$.


As shown by Takeuti~\cite{takeuti2015two}, there exists a canonical bijection  $\phi$ between $\R^{(\F)}\downarrow$ and $L^0$. 
 Moreover, the image of $\N^{(\F)}$ and $\Q^{(\F)}$ under $\phi$ are precisely $L^0(\N)$, the set of (equivalence classes) of $\N$-valued random variables; and $L^0(\Q)$, the set of (equivalence class) of $\Q$-valued random variables, respectively. 
 Besides, $\phi(r+s)=\phi(r)+\phi(s)$, $\phi(r s)=\phi(r)\phi(s)$, $\phi(0)=0$, $\phi(1)=1$,    
$\llbracket r=s\rrbracket=\bigvee\left\{A\in\F\colon 1_A\phi(r)=1_A\phi(s)\right\}$, and $\llbracket r\leq s\rrbracket=\bigvee\left\{A\in\F\colon 1_A\phi(r)\leq 1_A\phi(s)\right\}$,  for all $r,s\in\R^{(\F)}\downarrow$.

\begin{rem}
Gordon~\cite{gordon} proved that, in general, if $\A$ is an arbitrary complete Boolean algebra, the descent $\R^{(\A)}\downarrow$ is a universally complete vector lattice (i.e. every family of pairwise disjoint elements is bounded) such that $\A$ is isomorphic to the Boolean algebra of band projections in $\R^{(\A)}\downarrow$. 
Then, as a particular case, we find that $\R^{(\F)}\downarrow$ is isomorphic to the universally complete vector lattice $L^0$. 
Moreover, Takeuti~\cite{takeuti2015two} also proved that, in the case that $\A$ is the complete Boolean algebra of orthogonal projections in a Hilbert space, then $\R^{(\A)}\downarrow$ is isomorphic to the universally complete vector lattice of self-adjoint operators whose spectral resolution takes values in $\A$.  
\end{rem}


Suppose that $E$ is an $L^0$-module; that is, $E$ is a module over the ordered lattice ring $L^0$. 
We say that $E$ has the \emph{countable concatenation property} whenever for every sequence $\{x_k\}$ in $E$ and every partition $\{A_k\}\in p(\Omega)$ there exists a unique $x\in E$ (denoted by $x=\sum 1_{A_k}x_k$) such that $1_{A_k} x=1_{A_k} x_k$ for each $k\in\N$. 
This property and other related are technical assumptions that are typically assumed in literature cf.\cite{kupper03,guo10,zapata2017characterization}. 
 It should be pointed out that not every $L^0$-module has this property (for instance, see \cite[Example 2.12]{kupper03} and \cite[Example 1.1]{zapataportfolio}). 
 However, every $L^0$-module $E$ can be made into an $L^0$-module with this property by considering the quotient of a suitable equivalence relation on $E^{\N}\times p(\Omega)$ (see eg~\cite{OZ2017stabil}).\\ 

The next result describes the relation between $L^0$-modules and names for real vector spaces in $\VF$. 
Gordon~\cite{gordon1983rationally} provided an equivalence of categories between the category of names for real vector spaces and linear functions in $\VA$ and the category of unital separated injective $K$-modules and $K$-linear functions, where $K$ is a rationally complete semiprime commutative ring and $\A$ is the Boolean algebra of annihilator ideals (see~\cite{kusraev2012boolean} for terminology).  
It can be verified that $L^0$ is a rationally complete semiprime commutative ring whose annihilator ideals coincide with band projections.  
Besides, it can be checked that the countable concatenation property of an $L^0$-module $E$ is equivalent to the injectivity of $E$.   
Thus  Theorem \ref{thm: connectI} below is a particular case of the mentioned equivalence of categories in \cite{gordon1983rationally}. 
However, for the convenience of the reader, we provide a self-contained proof for this particular case.\footnote{In general, for any Boolean-algebra $\A$, one has that $\R^{(\A)}\downarrow$ is also a rationally complete semiprime commutative ring whose annihilator ideals coincide with bands, thus the equivalence of categories in~\cite{gordon1983rationally} also applies. For a proof of this particular case see~\cite[p. 198]{kusraev2012boolean}.}   

\begin{thm}
\label{thm: connectI}
For fixed an underlying measure algebra $\F$, there is an equivalence of categories between the category of names for real vector spaces and linear functions in $\VF$, and the category of $L^0$-modules with the countable concatenation property and $L^0$-linear functions.
\end{thm}

\begin{proof}
If we take any $E$ in $\VF$, which is a name for a real vector space, then $E\downarrow$ can be endowed with structure of $L^0$-module with the countable concatenation property. Indeed, for $x,y\in E\downarrow$ and $\eta\in L^0$, we define $x+y:=u$ where $u$ is the unique element of $E\downarrow$ such that $\llbracket x+y=u\rrbracket=\Omega$; we define $\eta x=v$, where $v$ is the unique element of $E\downarrow$ such that $\llbracket \phi^{-1}(\eta)x=v\rrbracket=\Omega$. 
It follows by inspection that $E\downarrow$ is an $L^0$-module. In addition, if $\{A_k\}\in p(\Omega)$ and $\{u_k\}\subset E\downarrow$, then by the mixing principle (Theorem \ref{thm: Mixing}) there exists a unique $u\in E\downarrow$ such that $A_k\leq\llbracket u=u_k\rrbracket$ for all $k\in\N$. Let $\tilde{1}_{A_k}:=\phi^{-1}(1_{A_k})$. One has $\llbracket\tilde{1}_{A_k}=1\rrbracket=A_k$ and $\llbracket\tilde{1}_{A_k}=0\rrbracket=A_k^c$ each $k$. Then
$$\llbracket \tilde{1}_{A_k}u=\tilde{1}_{A_k}u_k\rrbracket\geq \llbracket((u=u_k)\wedge (\tilde{1}_{A_k}=1))\vee (\tilde{1}_{A_k}=0)\rrbracket$$
$$=(\llbracket u=u_k\rrbracket\wedge \llbracket\tilde{1}_{A_k}=1\rrbracket)\vee \llbracket \tilde{1}_{A_k}=0\rrbracket =A_k\vee A_k^c=\Omega\quad\textnormal{ for all }k.$$ That is, $1_{A_k}u=1_{A_k}u_k$ for each $k$. 
This shows that $E\downarrow$ has also the countable concatenation property.

Now, suppose that $f,E,F$ are elements in $\VF$ such that  $E,F$ are names for real vector spaces and $f$ is a name for a linear function between $E$ and $F$ in $\VF$. 
Then, slightly abusing the notation, let $f\downarrow$ denote the unique map between $E\downarrow$ and  $F\downarrow$ satisfying $\llbracket f\downarrow(x)=f(x)\rrbracket=\Omega$ for all $x\in E\downarrow$. Then it can be verified that $f\downarrow:E\downarrow\rightarrow F\downarrow$ is an $L^0$-module morphism. 

Thus, we define the functor $G(E):= E\downarrow$, $G(f):= f\downarrow$. 

 Let us turn to the description of the inverse functor. 
Let $E$ be an $L^0$-module with the countable concatenation property. 
We will show that $E$ can be made into an element $\tilde{E}$ of $\VF$ which is a name for a real vector space. 
Indeed, for each $x\in E$ we define $\bar{x}:\mathcal{D}_x\rightarrow\F$ with
$$\quad\mathcal{D}_x:=\left\{\check{y}\colon y\in E \right\}\textnormal{, and }\bar{x}(\check{y}):=A_{x,y}\textnormal{ for }y\in E,$$
where $A_{x,y}:=\bigvee\left\{B\in\F\colon 1_B(x-y)=0\right\}$.  
Then, let $\tilde{E}:\mathcal{D}\rightarrow\F$ where 
$\mathcal{D}:=\left\{ \bar{x} \colon x\in E \right\}$ with $\tilde{E}(\bar{x}):=\Omega$ for each $x\in E$.

One has that $\tilde{E}$ is an element of $\VF$. 
Moreover, we claim that $\llbracket \bar{x}=\bar{y}\rrbracket=A_{x,y}$ for all $x,y\in E$. 
Indeed, given $x,y\in E$ one has
\begin{equation}
\label{eqI}
\llbracket \bar{x}=\bar{y}\rrbracket=\underset{u\in E}\bigwedge \left(A_{x,u}\Rightarrow\llbracket \check{u}\in\bar{y}\rrbracket \right)\wedge\underset{v\in E}\bigwedge \left(A_{y,v}\Rightarrow\llbracket \check{v}\in\bar{x}\rrbracket \right).
\end{equation}

In addition,
\[
\llbracket \check{u}\in\bar{y}\rrbracket=\underset{w\in E}\bigvee\left( A_{y,w}\wedge\llbracket \check{u}=\check{w}\rrbracket\right)=A_{y,u},
\]
since $\llbracket \check{u}=\check{w}\rrbracket=\Omega$ if $u=w$, and $\llbracket \check{u}=\check{w}\rrbracket=\emptyset$ otherwise.
 Analogously, we obtain $\llbracket \check{v}\in\bar{x}\rrbracket=A_{x,v}$.

Therefore, replacing in (\ref{eqI}), one has 
\[
\llbracket \bar{x}=\bar{y}\rrbracket=\underset{u\in E}\bigwedge \left(A_{x,u}^c\vee A_{y,u}\right) \wedge\underset{v\in E}\bigwedge \left( A_{y,v}^c\vee A_{x,v} \right).
\]

By considering above $u=x$ and $v=y$, it follows $\llbracket \bar{x}=\bar{y}\rrbracket\leq A_{x,y}$. 
Now, if we show that $A_{x,y}\leq A_{x,u}^c\vee A_{y,u}$ for each $u\in E$, we obtain the assertion. 
Aiming at a contradiction, suppose that $\emptyset< A:=A_{x,y}\wedge(A_{x,u}^c\vee A_{y,u})^c$. 
First, let us show that the supremum that defines $A_{x,y}$ is in fact attained for every $x,y\in E$. That is just to show that $1_{A_{x,y}}(x-y)=0$. Consider a maximal family $\mathfrak{M}$ of pairwise disjoint elements $B\in \mathcal{F}$ such that $1_B(x-y)=0$. By maximality, $\mathfrak{M}\in p(A_{x,y})$, and by countable chain condition this partition is countable. The uniqueness in the countable concatenation property yields that $1_{A_{x,y}}x=1_{A_{x,y}}y$.
 Finally, since $A\leq A_{x,y}\wedge A_{x,u}$ one has $1_A y=1_A x=1_A u$.  But $A\leq A_{y,u}^c$, hence  $1_A y\neq 1_A u$, which is a contradiction. 
 
For any $x\in E$, let $\tilde{x}$ be the representative in $\overline{V}^{(\F)}$ of the name $\bar{x}$. 
The function $E\rightarrow(\tilde{E})\downarrow$ given by $\:x\mapsto\tilde{x}$ is a bijection. 
Indeed, if $\tilde{x}=\tilde{y}$, then $\Omega=\llbracket \bar{x}=\bar{y}\rrbracket=A_{x,y}$; since $A_{x,y}$ is attained, it follows that $x=y$. 
Now, suppose that $z\in (\tilde{E})\downarrow$. Then 
$$\Omega=\llbracket z\in\tilde E\rrbracket=\underset{x\in E}\bigvee\llbracket \bar{x}=z\rrbracket.$$ 
We can find a partition $\{A_k\}\in p(\Omega)$ such that $A_k\leq \llbracket \bar{x}_k=z\rrbracket$ for some $x_k\in E$, each $k$. The countable concatenation property yields an $x$ so that $1_{A_k}x=1_{A_k}x_k$ for all $k$. 
One has $\llbracket\bar{x}=\bar{x}_k \rrbracket=A_{x,x_k}\geq A_k$. 
Due to the mixing principle we obtain $\llbracket z=\bar{x}\rrbracket=\Omega$, and taking representatives in $\overline{V}^{(\F)}$, we conclude that $z=\tilde{x}$.

For $x,y\in E$ and $r\in \R^{(\F)}\downarrow$ we  put $\llbracket \tilde{x}+\tilde{y}=u\rrbracket=\Omega$ whenever $\llbracket (x+y)^{\sim}=u\rrbracket=\Omega$ and $\llbracket z=r\cdot \tilde{x}\rrbracket=\Omega$ if $\llbracket (\phi(r)x)^{\sim}=z\rrbracket=\Omega$. Since the mapping $x\mapsto\tilde{x}$ is bijective and due to Theorem \ref{thm: extensional}, the operations are well-defined and $\tilde{E}$ is a name for a vector space in $\VF$. 

Now, suppose that $f:E\rightarrow F$ is a morphism between $L^0$-modules with the countable concatenation property. 
Then we define the application $g:\tilde{E}\downarrow\rightarrow\tilde{F}\downarrow$, $\tilde{x}\mapsto (f(x))^{\sim}$, which is well defined as $x\mapsto\tilde{x}$ is one-to-one. 
Using that $f$ is $L^0$-linear, we have that for every $x,y\in E$, $\llbracket \tilde{x}=\tilde{y}\rrbracket=A_{x,y}\leq A_{f(x),f(y)}=\llbracket (f(x))^{\sim}=(f(y))^{\sim}\rrbracket$. 
Then according to Theorem \ref{thm: extensional}, there exists $\tilde{f}$ in $\VF$ such that $\llbracket \tilde{f}:\tilde{E}\rightarrow\tilde{F}\rrbracket=\Omega$ and $\llbracket (f(x))^{\sim}=\tilde{f}(\tilde{x})\rrbracket=\Omega$ for all $x\in E$.
 In particular, we have that $\llbracket \tilde{f}(\tilde{x}+\tilde{y})=\tilde{f}(\tilde{x})+\tilde{f}(\tilde{y})\rrbracket=\Omega$ and $\llbracket \tilde{f}(r\cdot\tilde{x})=r\cdot \tilde{f}(\tilde{x})\rrbracket=\Omega$ for all $x,y\in E$ and $r\in \R^{(\F)}\downarrow$.

We define the functor $H(E):=\tilde{E}$, $H(f):=\tilde{f}$. 
Then the functors $F$ and $H$ are inverse equivalences. 
Indeed, suppose that $E$ is an $L^0$-module with the countable concatenation property. 
We have proved that $E\rightarrow(\tilde{E})\downarrow$, $x\mapsto\tilde{x}$ is a bijection. 
It is easy to verify that it is in fact an isomorphism of $L^0$-modules which defines a natural isomorphism between the functors $FH$ and the identity functor. 

 Also, given a name $E$ for a vector space in $\VF$, we consider the map $E\downarrow\rightarrow ((E\downarrow)^{\sim})\downarrow$, $x\mapsto\tilde{x}$. By applying Theorem \ref{thm: extensional}, we obtain a name for an isomorphism of vector spaces in $\VF$; that is, $\llbracket(E\downarrow)^{\sim}\cong E\rrbracket=\Omega$.
 Inspection shows that $HF$ is naturally isomorphic to the functor identity.
\end{proof}

Let us introduce some terminology:

If $E$ is an $L^0$-module with the countable concatenation property:
\begin{itemize}
\item $S\subset E$ is said to be:
\begin{enumerate}
\item \emph{$L^0$-convex}: if $\eta x + (1-\eta)y\in S$ for all $x,y\in S$ and $\eta\in L^0$ with $0\leq\eta\leq 1$;
\item \emph{$L^0$-absorbing}: if for every $x\in E$ there is $\eta\in L^0$, $\eta>0$, such that $x\in \eta x$;
\item \emph{$L^0$-balanced}: if $\eta x\in S$ whenever $x\in S$ and $\eta\in L^0$ with $|\eta|\leq 1$.
\end{enumerate}
\item A non-empty subset $S\subset E$ is said to be \emph{stable under countable concatenations}, or simply \emph{stable}, if for every countable family $\{x_k\}\subset S$ and partition $\{A_k\}\in p(\Omega)$, it holds that $\sum 1_{A_k}x_k\in S$.
\item A non-empty collection $\mathscr{C}$ of subsets of $E$ is called \emph{stable} if every $S\in \mathscr{C}$ is stable and for every countable family  $\{S_k\}\subset\mathscr{C}$ and partition $\{A_k\}\in p(\Omega)$, it holds that $\sum 1_{A_k}S_k\in \mathscr{C}$.
\end{itemize}

Filipovic et al.~\cite{kupper03} introduced the notion of locally $L^0$-convex module. 
Let us recall the following particular case, which was introduced in \cite{OZ2017stabil} and is a transcription in the present setting (via the equivalence of categories provided in \cite[Theorem 1.2]{OZ2017stabil}) of the notion of \emph{conditionally locally topological vector space} introduced in \cite{DJKK13}. 

\begin{defn} 
\label{def: stabL0mod}
A topological $L^0$-module $E[\mathscr{T}]$ with the countable concatenation property is said to be a \emph{stable locally  $L^0$-convex module} if there exists a neighborhood base $\mathscr{U}$ of $0\in E$ such that:
 \begin{itemize}
 \item[(i)] $\mathscr{U}$ is a stable collection;
 \item[(ii)] Every $U\in\mathscr{U}$ is $L^0$-convex, $L^0$-absorbing and $L^0$-balanced.
 \end{itemize}
 In this case, $\mathscr{T}$ is called a \emph{stable locally $L^0$-convex topology} on $E$.
\end{defn}

To our knowledge, the next result is new in literature; it describes the connection between names for locally convex spaces in $\VF$ and stable locally $L^0$-convex modules:

\begin{thm}
\label{thm: connectII}
For fixed an underlying measure algebra $\F$, there is an equivalence of categories between the category of names for locally convex spaces and continuous linear functions, and the category of stable locally $L^0$-convex modules and continuous $L^0$-module morphims.
\end{thm}

\begin{proof}

We consider the same functor $G$ as in Theorem \ref{thm: connectI}, 
but restricted to the category of names for locally convex spaces and continuous linear functions in $\VF$.
 
 Let $E[\mathcal{T}]$ be a name for a locally convex space in $\VF$; that is, $\mathcal{T}$ is a name for a locally convex topology in $\VF$. Let $\mathcal{U}$ be a name for a neighborhood base of the origin, such that 
 $$\llbracket\forall U\in\mathcal{U}(\textnormal{``}U\textnormal{ is convex'' }\wedge \textnormal{``}U\textnormal{ is absorbing'' }\wedge \textnormal{``}U\textnormal{ is balanced''})\rrbracket=\Omega.$$
 
 We know that $E\downarrow$ is an $L^0$-module with the countable concatenation property. 
 Let $\mathcal{U}\Downarrow:=\left\{ U\downarrow \colon U\in\mathcal{U}\downarrow\right\}$.
 Then every $U\in\mathcal{U}\Downarrow$ is stable due to the mixing principle. 
In the same way, again due to the mixing principle,  it holds that $\sum 1_{A_k}U_k\in\mathcal{U}\Downarrow$ whenever $\{U_k\}\subset\mathcal{U}\Downarrow$ and $\{A_k\}\in p(\Omega)$. 
Therefore, $\mathcal{U}\Downarrow$ is a stable collection. 
 Also, it is not difficult to show that each $U\in\mathcal{U}\Downarrow$ is $L^0$-convex, $L^0$-absorbing and $L^0$-balanced, and $\mathcal{U}\Downarrow$ is a neighborhood base of  $0\in E\downarrow$ of a topology $\mathscr{T}$.
 Therefore $E\downarrow[\mathscr{T}]$ is a stable locally $L^0$-convex module.

Now, let $f,E,F$ be elements of $\VF$ such that $E,F$ are names for locally convex spaces and $f$ is a name for a continuous linear function between $E$ and $F$ in $\VF$; that is, $\llbracket f\in L(E,F)\rrbracket=\Omega$. 
Then we can consider $f\downarrow$, which is an $L^0$-module morphism between $E\downarrow$ and $F\downarrow$ such that $\llbracket f\downarrow(x)=f(x) \rrbracket=\Omega$ for all $x\in E\downarrow$. The function $f\downarrow$ is also continuous. To see this, consider $U\downarrow$ a basic neighborhood of 0 in $F$. Then, since $f$ is continuous, there is a name for a basic neighborhood $W$ in $E$ such that $\llbracket W\subset f^{-1}(U)\rrbracket=\Omega$. 
Then $W\downarrow \subset f\downarrow^{-1}(U\downarrow)$ and this proves that $f\downarrow$ is continuous.

Conversely, let $E[\mathscr{T}]$ be a stable locally $L^0$-convex module. 
We consider $\tilde{E}$ as in the proof of Theorem \ref{thm: connectI}. 
Let $\mathscr{U}$ be a neighborhood base of $0\in E$ as in Definition \ref{def: stabL0mod}.
 For every $U\in\mathscr{U}$,  we can define $\tilde{U}:\mathcal{D}_U\rightarrow\F$, where $\mathcal{D}_{U}:=\left\{ \tilde{x} \colon x\in U\right\}$ and $\tilde{U}(\tilde{x}):=\Omega$. 
 Note that $\tilde{U}$ is an element of $\VF$. 
 The map $x\mapsto\tilde{x}$ gives a bijection between $U$ and $\tilde{U}\downarrow$. Injectivity was checked in the proof of Theorem~\ref{thm: connectI}. For surjectivity, if $w\in \tilde{U}\downarrow$ then $1 = \llbracket w\in\tilde{U}\rrbracket = \bigvee\{\llbracket w=\tilde{x}\rrbracket \colon x\in U\}$. Take a maximal family of pairs $A_k,x_k$ such that $x_k\in U$, the $A_k$ are nonzero and pairwise disjoint and $A_k \leq \llbracket w=\tilde{x}_k \rrbracket$. Using that $U$ is stable\footnote{In the sequel, we will omit the details of this usage of the countable concatenation property, that follows always the same scheme through a maximal disjoint family of nonzero elements.}, we obtained the desired preimage of $w$. 
 Now, we consider $\tilde{\mathscr{U}}:\mathcal{D}_{\mathscr{U}}\rightarrow\F$ 
 where $\mathcal{D}_{\mathscr{U}}:=\left\{ \tilde{U} \colon U\in\mathscr{U}\right\}$ and $\tilde{\mathscr{U}}(\tilde{U})=\Omega$.

 For each $U,V\in \mathscr{U}$, let $$A_{U,V}:=\bigvee\left\{A\in\F \colon 1_A U=1_A V\right\}.$$  
 We claim that $A_{U,V}=\llbracket \tilde{U}=\tilde{V}\rrbracket$. Indeed, 
 \[
 \llbracket \tilde{U}=\tilde{V}\rrbracket=\underset{x\in U}\bigwedge\llbracket \tilde{x}\in\tilde{V}\rrbracket \wedge \underset{y\in V}\bigwedge\llbracket \tilde{y}\in\tilde{U}\rrbracket=\underset{x\in U}\bigwedge\underset{y\in V}\bigvee A_{x,y} \wedge \underset{y\in V}\bigwedge\underset{x\in U}\bigvee A_{x,y}.
 \]

For every $x\in U$, by using that $V$ is stable similarly as above, one has that $1_{A_{U,V}}x=1_{A_{U,V}}y_x$ for some $y_x\in V$.
 Then $A_{U,V}\leq A_{x,y_x}$ for all $x\in U$. 
 Likewise, for every $y\in V$ we can find $x_y\in U$ with $A_{U,V}\leq A_{x_y,y}$.
  We conclude that  $A_{U,V}\leq \llbracket \tilde{U}=\tilde{V}\rrbracket$.

 By using again that $V$ is stable, for each $x\in U$ one can find $y^x\in V$ such that 
 $\bigvee_{y\in V} A_{x,y}=A_{x,y^x}$. 
 Similarly, for every $y\in V$ one can pick up $x^y\in U$ with $\bigvee_{x\in U} A_{x,y}=A_{x^y,y}$. 
 By using this in the expression above for $\llbracket \tilde{U}=\tilde{V}\rrbracket$, we get that always $\llbracket \tilde{U}=\tilde{V}\rrbracket \leq A_{x,y^x}$ and $\llbracket \tilde{U}=\tilde{V}\rrbracket \leq A_{y,x^y}$, and hence for every $x\in U$ we find that $1_{\llbracket \tilde{U}=\tilde{V}\rrbracket}x = 1_{\llbracket \tilde{U}=\tilde{V}\rrbracket}y^x$ and for every $y\in V$ we have that $1_{\llbracket \tilde{U}=\tilde{V}\rrbracket}x^y = 1_{\llbracket \tilde{U}=\tilde{V}\rrbracket}y$. 
 It follows that $1_{\llbracket \tilde{U}=\tilde{V}\rrbracket}U = 1_{\llbracket \tilde{U}=\tilde{V}\rrbracket}V$, and therefore $\llbracket \tilde{U}=\tilde{V}\rrbracket\leq A_{U,V}$.

 We proved before that the image of any $U\in\mathscr{U}$ via the mapping $x\mapsto\tilde{x}$ is $\tilde{U}\downarrow$. 
 Now we claim that the assignment $U\mapsto \tilde{U}\downarrow$ is a bijection from $\mathscr{U}$ to $\tilde{\mathscr{U}}\Downarrow$. The assignment is injective because $x\mapsto\tilde{x}$ is injective. It is surjective because if $W\downarrow\in \tilde{\mathscr{U}}\Downarrow$ with $W\in \tilde{\mathscr{U}}\downarrow$, then $\Omega = \llbracket W\in \mathcal{U}\rrbracket = \bigvee_{U\in \mathscr{U}}\llbracket W = \tilde{U}\rrbracket$, we can take a maximal disjoint family $\{B_k\}$ such that $B_k \leq \llbracket \tilde{U}_k=W\rrbracket$, and using the mixing principle and that $\mathscr{U}$ is a stable collection, we get that $V:=\sum 1_{A_k}U_k$ is a preimage for $W$. 
  
 Using that $\mathscr{U}$ is a neighborhood base of $0\in E$, it can be verified that $\tilde{\mathscr{U}}$ is a name for a  neighborhood base of the origin of a locally convex topology in $\VF$.
 
 Now, suppose that $f:E_1[\mathscr{T}_1]\rightarrow E_2[\mathscr{T}_2]$ is a continuous $L^0$-module morphism. Then we know that $\tilde{f}:\tilde{E}_1\rightarrow \tilde{E}_2$ is a name for a linear function. It is in fact a name for a continuous linear functional, because if a basic neighborhoods satisfy $f(U)\subset V$, then $\llbracket \tilde{f}(\tilde{U}) \subset \tilde{V}\rrbracket = \Omega$.
 
 Let $E[\mathscr{T}]$ be a stable locally $L^0$-convex module.  
 We know that the map $E\rightarrow \tilde{E}\downarrow$, $x\mapsto\tilde{x}$ is an isomorphism of $L^0$-modules. 
 Moreover, we proved before that $\mathscr{U}$ and $\tilde{\mathscr{U}}\Downarrow$ are one-to-one relation via $x\mapsto\tilde{x}$. 
 Consequently, $E[\mathscr{T}]\rightarrow (\tilde{E})\downarrow$, $x\mapsto\tilde{x}$ is also a homeomorphism.
 
 If $E[\mathcal{T}]$ is a name for a locally convex space, then $E[\mathcal{T}]\downarrow$ is a stable locally $L^0$-convex module and, by the argument above, the map $E[\mathcal{T}]\downarrow\rightarrow(  (E[\mathcal{T}]\downarrow)^{\sim}){\downarrow}$ is an isomorphism of $L^0$-modules which is also a homeomorphism. 
 Then, Theorem \ref{thm: extensional} provides a name for a homeomorphism between $E[\mathcal{T}]$ and $(E[\mathcal{T}]\downarrow)^{\sim}$. A close look shows that all these correspondences are natural transformations.
 
\end{proof}

The important conclusion of the result above is not the equivalence of categories itself, but that, for any stable locally $L^0$-convex module $E[\mathscr{T}]$, we can find a tailored name  $\tilde{E}[\mathcal{T}]$ for a locally convex space such that $E[\mathscr{T}]$ is isomorphic to the descent $\tilde{E}[\mathcal{T}]\downarrow$. 
This will allow to reinterpret certain objects related to $E[\mathscr{T}]$ within the Boolean-valued universe.


Henceforth, we will fix a stable locally $L^0$-convex module $E[\mathscr{T}]$ and its corresponding name $\tilde{E}[\mathcal{T}]$ for a vector space given by the equivalence of categories above. 
Since $E[\mathscr{T}]$ and $\tilde{E}[\mathcal{T}]\downarrow$ are isomorphic stable locally $L^0$-convex modules and the properties that we will study are preserved by the isomorphism, for simplicity, we will assume w.l.o.g. that one recovers the initial $L^0$-module by means of the descent, that is, $E[\mathscr{T}]=\tilde{E}[\mathcal{T}]\downarrow$. 
For the same reasons, we will assume that $L^0=\R^{(\F)}\downarrow$. 
Let $\bar{L^0}$ denote the set of equivalence classes of $\Sigma$-measurable functions with values in $[-\infty,+\infty]$ and let $\overline{\R}^{(\F)}$ be a name for the extended real numbers. 
Clearly, we can also assume $\bar{L^0}=\overline{\R}^{(\F)}\downarrow$. 

Next, we will list different relevant objects related to $E[\mathscr{T}]$. 
All of them are either introduced in the existing literature of $L^0$-convex analysis \cite{Cheridito2012,kupper03,guo10,OZ2017stabil} or come from transcriptions in the modular setting of elements of conditional set theory \cite{DJKK13,L0compactness,OZ2017stabil}. 
Also, some of these concepts came earlier from Boolean valued analysis as we will explain later in Remark \ref{rem: notionsBVA}. 
Our purpose is to discuss their meanings within Boolean-valued analysis, providing a bunch of 'building blocks' for the construction of module analogues of known statements of locally convex analysis, which will be also true due to the transfer principle:  
 
\begin{itemize}
	\item \emph{Stable subsets}: For a given stable subset $S$ of $E$ we define the name $\tilde{S}:\mathcal{D}_S\rightarrow\F$ where $\mathcal{D}_S:=\left\{ \tilde{x}\colon x\in S\right\}$ and $\tilde{S}(\tilde{x}):=\Omega$. 
	Then, $\tilde{S}$ is a name for a subset of $\tilde{E}$ with $\tilde{S}\downarrow=S$. 
	Conversely, if $S_0$ is a name with $\llbracket\emptyset\neq S_0\subset\tilde{E}\rrbracket=\Omega$, then $S_0\downarrow$ is a stable subset of $E$ satisfying $\llbracket(S_0\downarrow)^\sim=S_0\rrbracket=\Omega$. 
	
	Moreover, $S$ is $L^0$-convex if, and only if, $\llbracket \textnormal{``}\tilde{S}\textnormal{ is convex''}\rrbracket=\Omega$; $S$ is $L^0$-absorbing if, and only if, $\llbracket\textnormal{``} \tilde{S}\textnormal{ is absorbing''}\rrbracket=\Omega$; and $S$ is $L^0$-balanced if, and only if, $\llbracket \textnormal{``}\tilde{S}\textnormal{ is balanced''}\rrbracket=\Omega$.

	\item \emph{Stable collections of subsets}: For a given stable collection $\mathscr{C}$ of subsets of $E$ we define the name $\tilde{\mathscr{C}}:\mathcal{D}_{\mathscr{C}}\rightarrow\F$ with $\mathcal{D}_{\mathscr{C}}:=\left\{ \tilde{S} \colon S\in\mathscr{C}\right\}$ and $\tilde{\mathscr{C}}(\tilde{S}):=\Omega$.

	Conversely, if $\mathcal{C}$ is a name for a non-empty collection of non-empty subsets of $\tilde{E}$, we define $\mathcal{C}\Downarrow:=\{ S\downarrow \colon S\in \mathcal{C}\downarrow \}$, which is a stable collection of subsets of $E$.
	
	Moreover, if $\mathscr{C}$ is a stable collection of subsets of $E$, one has that $\mathscr{C}=\tilde{\mathscr{C}}\Downarrow$; and if $\mathcal{C}$ is a name for a non-empty collection of non-empty subsets of $\tilde{E}$, then we have $\llbracket(\mathcal{C}\Downarrow)^\sim=\mathcal{C}\rrbracket=\Omega$.  
	
	In particular, we know from the proof of Theorem \ref{thm: connectII} that, if $\mathscr{U}$ is a neighborhood base of $0\in E$ as in Definition \ref{def: stabL0mod}, then $\tilde{\mathscr{U}}$ is a name for a neighborhood base of the origin and
	\begin{equation}
	\label{eq: neighBas}
	\tilde{\mathscr{U}}\Downarrow=\mathscr{U}.
	\end{equation} 
	
	Let $\mathscr{C}$ be a stable collection of subsets of $E$. We denote by $(\cup \tilde{\mathscr{C}})_\F$, $(\cap \tilde{\mathscr{C}})_\F$ names for the union and intersection of $\tilde{\mathscr{C}}$, respectively. 
	Then, it holds that 
	\begin{equation}
	\label{eq:cupcap}
	\cup\mathscr{C}=(\cup \tilde{\mathscr{C}})_\F\downarrow\quad\textnormal{ and }\quad\cap \mathscr{C}=(\cap \tilde{\mathscr{C}})_\F\downarrow.
	\end{equation} 
	
	\item \emph{Stable open subsets}: Let $O\subset E$ be stable. 
	It follows from the relation (\ref{eq: neighBas}), that $O$ is open if and only if, $\llbracket\textnormal{``}\tilde{O}\textnormal{ is open''}\rrbracket=\Omega$. 
	
	Moreover, for any stable subset $S$ of $E$, one has $\llbracket(\text{int}(S))^\sim=\text{int}(\tilde{S})\rrbracket=\Omega$.  
	
	\item \emph{Stable closed subsets}: Let $C\subset E$ be stable. 
	Then relation (\ref{eq: neighBas}) allows to show that $C$ is closed if, and only if, $\llbracket\textnormal{``}\tilde{C}\textnormal{ is closed''}\rrbracket=\Omega$.  
	
	Moreover, for any stable subset $S$ of $E$, it holds $\llbracket(\text{cl}(S))^\sim=\text{cl}(\tilde{S})\rrbracket=\Omega$. 
	
	\item  \emph{Stable filters}: A \emph{stable filter} on $E$ is a filter $\mathscr{F}$, which admits a filter base $\mathscr{B}$ which is a stable collection of subsets of $E$. 
	
	If $\mathscr{F}$ is a stable filter with base $\mathscr{B}$, where $\mathscr{B}$ is a stable collection, then it can be verified that $\llbracket\textnormal{``}\tilde{\mathscr{B}}\textnormal{ is a filter base''}\rrbracket=\Omega$. 
	Conversely, if $\mathcal{B}$ is a name for a filter base, then $\mathcal{B}\Downarrow$ is the base of some stable filter. 
	
	\item \emph{Stably compact subsets}: A stable subset $S$ of $E$ is said to be \emph{stably compact}, if every stable filter base $\mathscr{B}$ on $S$ has a cluster point in $S$. 
	
	We have that $K\subset E$ is stably compact if, and only if, $\llbracket\textnormal{``}\tilde{K}\textnormal{ is compact''}\rrbracket=\Omega$. 
	This follows because, due to (\ref{eq: neighBas}), a stable filter base $\mathscr{B}$ has a cluster point if, and only if,  $\llbracket\textnormal{``} \tilde{\mathscr{B}}\textnormal{ has a cluster point''}\rrbracket=\Omega$; and a name for a filter base $\mathcal{B}$ satisfies $\llbracket\textnormal{``}\mathcal{B}\textnormal{ has a cluster point''}\rrbracket=\Omega$ if, and only if, $\mathcal{B}\Downarrow$ has a cluster point.
	
	We will say that a stable subset $K$ of $E$ is \emph{relatively stably compact}, if $\textnormal{cl}(K)$ is stably compact. 
	Notice that $K\subset E$ is relatively stably compact if, and only if, $\llbracket\textnormal{``}\tilde{K}\textnormal{ is relatively compact''}\rrbracket=\Omega$.
	
	Just mention that it was proven in \cite[Proposition 5.2]{L0compactness} that, when the underlying probability space is atomless, then any Hausdorff stable locally $L^0$-convex module is anti-compact; that is, the only compact subsets are the finite subsets. 
	This means that the conventional compactness is not interesting because does not allow to establish any meaningful theorem. 
	On the other hand, the transfer principle brings a huge range of theorems involving stable compactness, which shows that stable compactness is by far much richer than classical compactness. 
	
	\item \emph{Stable functions}: Suppose that $S_1,S_2\subset E$ are stable. 
	A function $f:S_1\rightarrow S_2$ is said to be \emph{stable} if $f(\sum 1_{A_k}x_k)=\sum 1_{A_k}f(x_k)$ for all $\{x_k\}\subset S_1$ and $\{A_k\}\in p(\Omega)$.
	
	Since $\llbracket x=y\rrbracket\leq \llbracket f(x)=f(y)\rrbracket$ for all $x,y\in S_1$,  
	 Theorem \ref{thm: extensional} yields a name for a function $\tilde{f}$ between $\tilde{S}_1$ and $\tilde{S}_2$ such that 
	$\tilde{f}\downarrow=f$. 
	
	Moreover, it can be verified that $f$ is continuous if, and only if, $\llbracket\textnormal{``}\tilde{f}\textnormal{ is continuous''}\rrbracket=\Omega$.     
	
	A function $f:E\rightarrow \bar{L^0}$ has the \emph{local property}, if $1_A f(x)=1_A f(1_A x)$ for all $A\in\F$. 
	If $f$ has the local property, once again Theorem \ref{thm: extensional} allows to define a name $\tilde{f}$ of a function from $\tilde{E}$ to $\overline{\R}^{(\F)}$ so that $\tilde{f}\downarrow=f$.
	
	A function $f:E\rightarrow \bar{L^0}$ is:
\begin{enumerate}
	\item \emph{$L^0$-convex}: if $f(\eta x + (1-\eta) y)$ for all $\eta\in L^0$ with $0\leq\eta\leq 1$ and $x,y\in E$;
	\item \emph{proper}: if $f(x)>-\infty$ for all $x\in E$ and there is some $x_0\in E$ with $f(x_0)\in L^0$;
	\item \emph{lower semi-continuous}: if the sublevel $V_f(\eta):=\left\{x\in E \colon f(x)\leq\eta\right\}$
	is closed for every $\eta\in\bar{L^0}$.
\end{enumerate}
The \emph{domain} of $f$ is defined by $\textnormal{dom}(f):=\left\{ x\in E \colon f(x)\in L^0\right\}.$

	When $f$ has the local property, one has that $f$ is $L^0$-convex if, and only if $\llbracket\textnormal{``}\tilde{f}\textnormal{ is convex''}\rrbracket=\Omega$; and $f$ is proper if, and only if, $\llbracket\textnormal{``}\tilde{f}\textnormal{ is proper''}\rrbracket=\Omega$. 
 Further, it can be verified that $V_f(\eta)$ is a stable set for each $\eta\in L^0$ such that $V_f(\eta)\neq\emptyset$. 
Thus, we can conclude that, $f$ is lower semi-continuous  if, and only if, $\llbracket\textnormal{``}\tilde{f}\textnormal{ is lower semi-continuous''}\rrbracket=\Omega$. 
	
	Finally, just mention that if $f$ is $L^0$-convex, then $f$ has automatically the local property (see \cite[Theorem 3.2]{kupper03}), hence in the statements we will not have to require the latter property.
	
	\item \emph{Topological dual}: We consider $E^\ast:=E^\ast[\mathscr{T}]$ the set of all continuous $L^0$-module morphisms $\mu:E\rightarrow L^0$. 
	Then, we can consider the name $F:\mathcal{D}_{E^\ast}\rightarrow\F$ with $\mathcal{D}_{E^\ast}:=\{ \tilde{\mu} \colon \mu\in E^\ast\}$ and $F(\tilde{\mu}):=\Omega$. 
	Then we have $\llbracket F=\tilde{E}^\ast[\mathcal{T}]\rrbracket=\Omega$, where $\tilde{E}^\ast[\mathcal{T}]$ denotes a  name for the topological dual of $\tilde{E}[\mathcal{T}]$. 
	Moreover, note that we have the relation $E^\ast[\mathscr{T}]=\{\mu\downarrow \colon \mu\in\tilde{E}^\ast[\mathcal{T}]\downarrow\}$. 
	
	\item \emph{Stable sequences}: A net $\chi=\{x_{\mathfrak{n}}\}_{\mathfrak{n}\in L^0(\N)}$ in $E$ is called a \emph{stable sequence} whenever $x_{\mathfrak{n}}=\sum_{k\in\N}1_{\{\mathfrak{n}=k\}}x_k$ for all $\mathfrak{n}\in L^0(\N)$. 
	Then, again, Theorem \ref{thm: extensional} provides us with a name $\tilde{\chi}$ for a function from $\N^{(\F)}$ to $\tilde{E}$; that is, a name for a sequence in $\tilde{E}$. Besides, we have $\tilde{\chi}\downarrow=\chi$. 
	
	Bearing in mind relation (\ref{eq: neighBas}), it can be verified that the net $\chi$ converges to $x\in E$ if, and only if, $\llbracket\textnormal{``}\tilde{\chi}\textnormal{ converges to }\tilde{x}\in\tilde{E}\textnormal{''}\rrbracket=\Omega$.    
	
	A stable sequence $\kappa=\{y_{\mathfrak{n}}\}_{\mathfrak{n}\in L^0(\N)}\subset E$ is called a \emph{stable subsequence} of $\chi=\{x_{\mathfrak{n}}\}_{\mathfrak{n}\in L^0(\N)}$ if there exists a stable sequence $\{\mathfrak{n}_{\mathfrak{m}}\}_{\mathfrak{m}\in L^0(\N)}\subset L^0(\N)$, with ${\mathfrak{n}}_{\mathfrak{m}}<\mathfrak{n}_{\mathfrak{m}^\prime}$ whenever $\mathfrak{m}<\mathfrak{m}^\prime$, such that $y_{\mathfrak{m}}=x_{\mathfrak{n}_{\mathfrak{m}}}$ for all $\mathfrak{m}\in L^0(\N)$. 
	In this case, it can be verified that $\llbracket \textnormal{``}\tilde{\kappa}\textnormal{ is a subsequence of }\tilde{\chi}\textnormal{''}\rrbracket=\Omega$. 
	
	\item \emph{$L^0$-norms}: An \emph{$L^0$-norm} on $E$ is a function $\Vert\cdot\Vert:E\rightarrow L^0$ such that for all $x,y\in E$ and $\eta\in L^0$ satisfies: 
	\begin{itemize}
		\item[(i)] $\Vert x\Vert\geq 0$, with $\Vert x\Vert=0$ if and only if $x=0$;
		\item[(ii)] $\Vert \eta x\Vert=|\eta|\Vert x\Vert$;
		\item[(iii)] $\Vert x+y\Vert\leq\Vert x\Vert+\Vert y\Vert$.
	\end{itemize}
	In this case $(E,\Vert\cdot\Vert)$  is called an $L^0$-\emph{normed module}. 
	
	The collection of sets $B_\varepsilon:=\{x\in E\colon \Vert x\Vert<\varepsilon\}$, where $\varepsilon\in L^0$ with $\varepsilon>0$, is a neighborhood base of $0\in E$ for a stable locally convex topology $\mathscr{T}$. 
	
	Due to (ii), $\Vert\cdot\Vert$ has the local property. Then we can define $\Vert\cdot\Vert^\sim$, which is a name for a norm on $\tilde{E}$. 
	Furthermore, one has that $\llbracket\textnormal{``}\Vert\cdot\Vert^\sim\textnormal{ induces }\mathcal{T}\textnormal{''}\rrbracket=\Omega$.   
	
%
	
	\item \emph{Stable completeness}: 
	Suppose that $(E,\Vert\cdot\Vert)$ is an $L^0$-normed module. 
	A stable sequence $\{x_{\mathfrak{n}}\}_{\mathfrak{n}\in L^0(\N)}\subset E$ is said to be \emph{Cauchy} if for every $\varepsilon\in L^0$, $\varepsilon>0$, there exists $\mathfrak{n}_0\in L^0(\N)$ such that $\Vert x_{\mathfrak{n}}-x_{\mathfrak{n}^\prime}\Vert\leq \varepsilon$ for all $\mathfrak{n},\mathfrak{n}^\prime\in L^0(\mathbb{N})$ with $\mathfrak{n},\mathfrak{n}^\prime\geq \mathfrak{n}_0$.
	
	We say that $(E,\Vert\cdot\Vert)$ is \emph{stably complete}, if every Cauchy stable sequence is convergent.
	
	Then, $(E,\Vert\cdot\Vert)$ is stably complete if, and only if, $\llbracket\textnormal{``}(\tilde{E},\Vert\cdot\Vert^\sim)\textnormal{ is a Banach space''}\rrbracket=\Omega$.

	\item \emph{Stable weak topologies}: The collection of sets
	\[
	U_{\{F_k\},\{A_k\},\varepsilon}:=\{x\in E\colon \sum 1_{A_k}\underset{\mu\in F_k}\esssup|\mu(x)|<\varepsilon\},
	\]
	where $\{A_k\}\in p(\Omega)$, $\{F_k\}$ is a countable collection of non-empty finite subsets of $E^\ast$ and $\varepsilon\in L^0$ with $\varepsilon>0$, is a neighborhood base of $0\in E$ for a stable locally $L^0$-convex topology, which is called the \emph{stable weak topology} and is denoted by $\sigma_s(E,E^\ast)$. 
	
	Then the corresponding name for a locally convex topology provided by the equivalence of categories in Theorem \ref{thm: connectII} is precisely a name for the weak topology of $\tilde{E}[\mathcal{T}]$.
	
	Analogously, we can define the \emph{stable weak-$\ast$ topology} $\sigma_s(E^\ast,E)$.
\end{itemize}

\begin{rem}\label{rem: notionsBVA}
As mentioned previously, some of the notions listed above  were introduced earlier in literature of Boolean-valued analysis under different nomenclature. 
\emph{Stable compactness} was formulated in \cite{L0compactness} as a transcription of the notion of \emph{conditional compactness} introduced in \cite{DJKK13}. 
However, stable compactness was first time studied by Kusraev~\cite{cyclicCompactness} giving rise to the notion of \emph{cyclic compact set}. 
Later, the notion of \emph{mix-compactness} was introduced by Gutman and Lisovskaya \cite{gutman2009boundedness}.     
It turns out that \emph{cyclic compactness} and \emph{mix-compactness} are equivalent notions (see~\cite[Theorem 2.12.C.5]{kusraev2014boolean}). 
These types of compactness have been fruitfully exploited, see for instance results in \cite[Sections 1.3 and 1.4]{kusraev1985vector},  \cite[Section 8.5]{kusraev2000dominated} and the analogues of the boundedness and uniform boundedness  principles  obtained in \cite{gutman2009boundedness}.

The notion of \emph{stable completeness} is  a transcription of the notion of \emph{conditional completeness} introduced in \cite{DJKK13}. 
Descents of complete spaces and Banach spaces were studied earlier by Kusraev \cite{kusraev1985banach}, originating the notion of Banach-Kantorovich space, which are descents of real Banach spaces as proven in \cite{kusraev1985banach} (for further details see~\cite[Section 8.3]{kusraev2000dominated} and \cite[Section 5.4]{kusraev2012boolean}). 

Finally, the \emph{stably weak and stably weak-$\ast$ topologies} defined above are transcriptions of the notion of \emph{conditional initial topology} induced by \emph{conditional dual pairs} introduced in \cite{DJKK13} applied to the pairing $\langle E,E^\ast\rangle$. 
Descents of dual pairs, which give rise to dual systems with $\R^{(\A)}\downarrow$-bilinear forms, were studied earlier in \cite{kusraev1985vector}. In particular, \cite[Theorem 3.3.10(b)]{kusraev1985vector} is related to Theorem \ref{thm: connectII}. 
This type of pairings covers the \emph{stably weak and stably weak-$\ast$ topologies} defined above in the more general framework of modules over universally complete ring lattices.
\end{rem}

Once we have the 'building blocks', let us see some examples to exhibit how they can be assembled to give rise to different statements. 
Of course, this list is not exhaustive and we can create many other pieces for our puzzle.
 
Let us start by the main theorems of \cite{kupper03}. 
For instance, we will see that Theorems 2.8, 3.7 and 3.8 follow from the transfer principle of Boolean-valued models. 

Although, these results apply to the more general structure of locally $L^0$-convex module, they are proved under the assumption that the locally $L^0$-convex topology is induced by a family of $L^0$-seminorms (see \cite[Definition 2.3]{kupper03}), which is closed under  finite suprema and with the so-called \emph{countable concatenation property}.\footnote{Here, we refer to the countable concatenation property for families of $L^0$-seminorms, which has not to be missed up with the algebraic countable concatenation property introduced at the beginning of the section.} 
It is not difficult to prove that these properties amount to the existence of a neighborhood base $\mathscr{U}$ of $0\in E$ as in Definition \ref{def: stabL0mod}. Thus, these results implicitly apply to stable locally $L^0$-convex modules.  


We have the following:\footnote{This statement is more general than \cite[Theorem 2.8]{kupper03} as the latter applies to the particular case in which $S_1$ is a singleton. This statement is also a transcription of \cite[Theorem 5.5(ii)]{DJKK13} as shown in \cite{L0compactness}.}  

\begin{thm}\label{thm: separation}
	Let $E[\mathscr{T}]$ be a  stable locally $L^0$-convex module, and suppose that $S_1,S_2$ are stable and $L^0$-convex subsets with $S_1$ stably compact and $S_2$ closed. 
	If
	$$1_A S_1\cap 1_A S_2=\emptyset\quad\textnormal{ for all }A\in\F\textnormal{ with }A>\emptyset,$$  
	then there exists a continuous $L^0$-module morphism $\mu:E\rightarrow L^0$ and $\varepsilon\in L^0$, $\varepsilon>0$, such that 
	$$\mu(x)>\mu(y)+\varepsilon\quad\textnormal{ for all }x\in S_1,y\in S_2.$$
\end{thm}

Remember the classical separation theorem: If $C,K$ are non-empty convex subsets with $C$ closed, $K$ compact, and $C$ and $K$ have empty intersection, then there is a lineal functional that separates $C$ from $K$. 
What we have above is just a reformulation of the statement $\llbracket \text{separation theorem}\rrbracket = \Omega$, so no proof needed.\\ 

In literature, there is a long tradition of studying \emph{conjugates} and \emph{subgradients} of functions taking values in different types of ordered lattice rings such as Kantorovich spaces (see eg~\cite[chap. 4]{kusraev2012subdifferentials}), and addressing versions of the classical Fenchel-Moreau theorem in these settings (see eg~\cite[Theorem 4.3.10(1)]{kusraev2012subdifferentials} and \cite[Theorem 1.2.11]{kusraev1985vector}).    
More recently, Filipovic et al~\cite{kupper03} worked with versions of conjugates and subgradients for $\bar{L^0}$-valued functionals defined on $L^0$-modules.   
Namely, the \emph{conjugate} of a function $f:E\rightarrow\bar{L^0}$ is defined by
\[
f^\ast:E^\ast\rightarrow\bar{L^0},\quad f^\ast(\mu):=\underset{x\in E}\esssup(\mu(x)-f(x)),
\]
and its \emph{biconjugate} is defined by
\[
f^{\ast\ast}:E\rightarrow\bar{L^0},\quad f^{\ast\ast}(x):=\underset{\mu\in E^\ast}\esssup(\mu(x)-f^\ast(\mu)).
\] 
An element $\mu\in E^\ast$ is a \emph{subgradient} of $f:E\rightarrow\bar{L^0}$ at $x_0\in\text{dom}(f)$, if 
\[
\mu(x-x_0)\leq f(x)-f(x_0)\quad\text{ for all }x\in E.
\] 
The set $\partial f(x_0)$ stands for the set of all subgradients of $f$ at $x_0$. 

The notion of $L^0$-barrel was introduced in \cite{kupper03}. 
Namely, a subset $S$ of $E$ is an \emph{$L^0$-barrel} if it is $L^0$-convex, $L^0$-absorbing, $L^0$-balanced and closed. 
We will say that a topological $L^0$-module is \emph{stably barreled} if every stable $L^0$-barrel is a neighborhood of $0\in E$.\\

\cite[Theorem 3.8]{kupper03} is a module analogue of the classical Fenchel-Moreau theorem.  
We have the following statement, which does not need a proof as it follows from its conventional version \cite[Theorem 2.22]{barbu2012convexity} by means of the transfer principle by just noting that $\llbracket \tilde{f}^{\ast\ast}=(f^{\ast\ast})^\sim\rrbracket=\Omega$:  

\begin{thm}\label{thm: fenchel}
	Let $E[\mathscr{T}]$ be a stable locally $L^0$-convex module and let $f:E\rightarrow\bar{L^0}$ be proper lower semi-continuous and $L^0$-convex. 
	Then $f^{\ast\ast}=f$.
\end{thm}

Concerning subgradients, we have the following result, which is a generalization of \cite[Theorem 3.7]{kupper03} and follows from the transfer principle applied to the so-called Fenchel-Rockafellar theorem, see eg~\cite[Theorem 1]{borwein2006variational}:

\begin{thm}\label{thm: subGrad}
	Let $E[\mathscr{T}]$ be a stable locally $L^0$-convex module which is stably barreled. 
	Let $f:E\rightarrow\bar{L^0}$ be a proper lower semicontinuous $L^0$-convex function. Then, 
	\[
	\partial f(x)\neq \emptyset\quad\text{ for all }x\in\textnormal{int}(\textnormal{dom}(f)).
	\]
\end{thm}

The notion of \emph{$L^0$-barreled} topological $L^0$-module	was introduced in \cite{kupper03}; namely, $E[\mathscr{T}]$ is $L^0$-barreled if every $L^0$-barrel is a neighborhood of $0\in E$. 
	Thus, the notion of stably barreled topological $L^0$-module is more general. 
	This was already pointed out in \cite{guo2015random1}, where the statement above was already proven by using the techniques introduced in \cite{kupper03}.\\

Let us see more examples of application of our method.   
Next, we provide module analogues of the classical James' compactness theorem and also a version of the important Brouwer fixed point theorem.

The following statement is a modular version of a non-linear variation of classical James' compactness theorem, which plays an important role in the study of robust representation of risk measures (see eg~\cite[Theorem A.1]{jouini2006law} and \cite[Theorem 2]{orihuela2012coercive}). 
The statement we present follows from the transfer principle applied to its most general version \cite[Theorem 2.4]{saint2013weak}.

\begin{thm}\label{thm: nonLinJames}
	Let $(E,\Vert\cdot\Vert)$ be a stably complete $L^0$-normed module and let $f:E\rightarrow \bar{L^0}$ be a proper function with the local property. 
	If for every $\mu\in E^\ast$ there is an $x_0$ such that $\mu(x_0)-f(x_0)=f^\ast(\mu)$, then the set $V_f(\eta)=\{x \in E \colon f(x)\leq \eta\}$ is relatively stably compact w.r.t. $\sigma_s(E,E^\ast)$ for every $\eta\in L^0$ with $V_f(\eta)\neq\emptyset$. 
\end{thm}

Of course, we also have a modular version of the celebrated James' compactness theorem, which is a consequence of the statement above, and also follows from the transfer principle applied to its classical version:

\begin{thm}
	Let $(E,\Vert\cdot\Vert)$ be a stably complete $L^0$-normed module and let $K\subset E$ be stable, $L^0$-convex and $L^0$-norm bounded (i.e. $\esssup_{x\in K}\Vert x\Vert<\infty$). 
	Then, $K$ is stably compact w.r.t. $\sigma_s(E,E^\ast)$ if, and only if, each $\mu\in E^\ast$ there exists $x_0\in K$ such that $\mu(x_0):=\esssup_{x\in K} \mu(x)$.
\end{thm}

A version of the Brouwer Fixed Point Theorem for $(L^0)^d$ was provided in \cite{DKKM13}, which corresponds to the finite-dimensional case in our context.  
Next, we will state a Brouwer fixed point theorem for Hausdorff\footnote{In view of (\ref{eq: neighBas}) and (\ref{eq:cupcap}), It is not difficult to show that $E[\mathscr{T}]$ is Hausdorff if, and only if, $\bigcap\mathscr{U}=\{0\}$, if, and only if, $\llbracket \bigcap\tilde{\mathscr{U}}=\{0\}\rrbracket=\Omega$ and if, and only if, $\llbracket\textnormal{``} \tilde{E}[\mathcal{T}]\textnormal{ is Hasdorff''}\rrbracket=\Omega$} stable locally $L^0$-convex modules, which is a direct application of the transfer principle to the so-called Schauder-Tychonov Theorem.

\begin{thm}
	If $S$ is an $L^0$-convex and stably compact subset of a Hausdorff stable locally $L^0$-convex module $E[\mathscr{T}]$, then any stable and continuous function $f:S\rightarrow S$ has a fixed point in $S$.
\end{thm}

Obviously, all these Theorems  are just some examples: we can state a version of any theorem $T$ on locally convex spaces and it immediately renders a version for locally $L^0$-modules of the form $\llbracket T \rrbracket = \Omega$.\\

Finally, let us turn to the discussion of an example of financial application:

The notion of \emph{convex risk measure} was independently introduced by Föllmer and Schied \cite{follmer2002convex} and Fritelli and Gianin \cite{frittelli2002putting} as an extension of the notion of \emph{coherent risk measure} introduced in  Artzner et al.~\cite{artzner01}. 
Let $\mathscr{X}$ be an ordered vector space with $\R\subset \mathscr{X}$ which models all the financial positions in a financial market.  
A convex risk measure is a proper convex function $\rho:\mathscr{X}\rightarrow\overline{\R}$ which satisfies the following conditions for all $x,y\in\mathscr{X}$:
\begin{itemize}
	\item \emph{Monotonicity}: if $x\leq y$, then $\rho(y)\leq\rho(x)$;
	\item \emph{Cash invariance}: $\rho(x+r)=\rho(x)-r$, for all $r\in\R$.
\end{itemize}

Now, suppose that $(\Omega,\Sigma,\PP)$ models the market events at some future date $t>0$. 
In this case, from a modelling point of view, the risk  of any financial position is contingent on the  information encoded in the measure algebra $\F$. 
For instance, the risk measurably depends on the decisions taken by the risk manager in virtue of the market eventualities arisen at time $t$.   
Therefore, in this case, the different financial positions can be modelled by an ordered $L^0$-module $\mathscr{X}$  with $L^0\subset \mathscr{X}$. 
Filipovic et al.~\cite{kupper11} proposed the following definition: a \emph{conditional convex risk measure} is a proper $L^0$-convex function $\rho:\mathscr{X}\rightarrow \bar{L^0}$ which satisfies the following conditions for all $x,y\in\mathscr{X}$:
\begin{itemize}
	\item \emph{Monotonicity}: if $x\leq y$, then $\rho(y)\leq\rho(x)$;
	\item \emph{Cash invariance}: $\rho(x+\eta)=\rho(x)-\eta$, for all $\eta\in L^0$.
\end{itemize}  

Since a conditional convex risk measure $\rho:\mathscr{X}\rightarrow \bar{L^0}$ is $L^0$-convex, in particular, it has the local property, and Theorem \ref{thm: extensional} defines a name for a function $\tilde{\rho}$ from $\tilde{\mathscr{X}}$ to $\overline{\R}^{(\F)}$. 
Moreover, it can be verified that $\llbracket\textnormal{``}\tilde{\rho}\textnormal{ is convex, monotone and cash-invariant''}\rrbracket=\Omega$. 
We conclude that a conditional convex risk measure $\rho$ can be identified with a name $\tilde{\rho}$ for a convex risk measure  within $V^{(\F)}$.  
Thus, the machinery of Boolean-value models and its transfer principle can be applied. 

From a modelling point of view, we have that, in the same manner the available market information is encoded in $\F$,
the financial strategy followed by the risk manager in order to maximize or hedge future payments  can be analytically expressed in terms of the formal language $\mathcal{L}^{(\F)}$,
 which consistently depends on the information of $\F$. 
 Thus the Boolean-valued analysis makes available to us a powerful technology to incorporate trading rules based on equilibrium prices or risk constraints in the mathematical analysis of certain problems of mathematical finance involving a multi-period setting.


%
%
%

\section{A precise connection between Conditional set theory and Boolean-valued models}

In \cite{DJKK13} it was introduced the notion of conditional set:

\begin{defn}
\cite[Definition 2.1]{DJKK13}
\label{def: condSet}
Let $X$ be a non-empty set and let $\A$ be a complete Boolean algebra. 
A \emph{conditional set} of $X$ and $\A$ is a set $\textbf{X}$ such that there exists a surjection $(x,a)\mapsto x|a$ from $X\times \A$ onto $\textbf{X}$ satisfying: 

\begin{enumerate}
	\item[(C1)] if $x,y\in X$ and $a,b\in\F$ with $x|a=y|b$, then $a=b$;
	\item[(C2)] (Consistency) if $x,y\in X$ and $a,b\in\A$ with $a\leq b$, then $x|b=y|b$ implies $x|a=y|a$;
	\item[(C3)] (Stability) if $\{a_i\}_{i\in I}\in p(1)$ and $\{x_i\}_{i\in I}\subset X$, then there exists a unique $x\in X$ such that $x|a_i=x_i|a_i$ for all $i\in I$. 
\end{enumerate}
    
The unique element $x\in X$ provided by C3, is called the \emph{concatenation} of the family $\{x_i\}$ along the partition $\{a_i\}$, and is denoted by  $\sum x_i|a_i$.
\end{defn}  


Let $\textbf{X},\textbf{Y}$ be conditional sets. According to \cite[Definition 2.1]{DJKK13} a function $f:X\rightarrow Y$ is said to be \emph{stable} if 
$$f\left(\sum x_i|a_i\right)=\sum f(x_i)|a_i,\quad\textnormal{ for }\{a_i\}\in p(1),\:\{x_i\}\subset X.$$
 
If $f:X\rightarrow Y$ is a stable function, it is simply to verify that $\textbf{G}_{\textbf{f}}:=\left\{(x|a,f(x)|a)\;\:\; x\in X,\:a\in\A\right\}$ 
is a conditional set of the graph of $f$ and $\A$. $\textbf{G}_{\textbf{f}}$ is called the \emph{conditional graph of a conditional function} $\textbf{f}:\textbf{X}\rightarrow\textbf{Y}$ (see \cite[Definition 2.1]{DJKK13}). 

A conditional function $\textbf{f}:\textbf{X}\rightarrow\textbf{Y}$ is \emph{conditionally injective} if $x|a\neq x^\prime|a$ for all $a>0$ implies that $f(x)|a\neq f(x^\prime)|a$ for all $a>0$; it is \emph{conditionally surjective} whenever $f$ is surjective; and it is a \emph{conditional bijection} if it is conditionally injective and surjective.

Then the following result gives the relation between conditional sets of the universe $V$ and the boolean-valued universe $\VA$.

\begin{thm}\label{thm: connectIII}
For fixed a Boolean algebra $\A$, there is an equivalence of categories between the category of conditional sets of $\A$ whose morphisms are conditional functions, and the category of elements $x$ of $\VA$ such that $\llbracket x\neq\emptyset\rrbracket=1$ whose morphisms are names for functions in $\VA$.   
\end{thm} 

\begin{proof}
First, suppose that $x$ is an element of $\VA$ with $\llbracket x\neq\emptyset\rrbracket=1$. 
Then we consider the equivalence relation on $(x\downarrow)\times\A$ given by
$$(u,a)\sim(v,b)\quad\textnormal{ whenever }a=b,\: \llbracket u=v\rrbracket\geq a.$$

Let us denote by $x|a$ the class of $(x,a)$ and let $\textbf{x}\downarrow$ be the corresponding quotient set. 
Then $\textbf{x}\downarrow$ is a conditional set of $x\downarrow$ and $\A$. 
Indeed, (C1) and (C2) from Definition \ref{def: condSet} are trivially satisfied. 
Further, (C3) follows from the mixing principle (Theorem \ref{thm: Mixing}). 

Suppose that $f,X,Y$ are in $\VA$ and $\llbracket f:X\rightarrow Y\rrbracket=1$.
Then $f\downarrow$ is a function from $X\downarrow$ to $Y\downarrow$ such that $\llbracket f\downarrow(u)=f(u)\rrbracket=1$ for all $u\in X\downarrow$.
Now, we claim that $f\downarrow:X\downarrow\rightarrow Y\downarrow$ is a stable function of the conditional sets $\textbf{X}\downarrow$, $\textbf{Y}\downarrow$.  
Indeed, given $\{a_i\}\in p(1)$ and $\{u_i\}\subset X$ we take $u:=\sum u_i|a_i\in x\downarrow$.
 We have that $\llbracket f\downarrow(u)=f\downarrow(u_i)\rrbracket=\llbracket f(u)=f(u_i)\rrbracket\geq \llbracket u=u_i\rrbracket\geq a_i$, and thus $f\downarrow(u)|a_i=f\downarrow(u_i)|a_i$  each $i$.
 This shows that $f\downarrow(u)=\sum f\downarrow(u_i)|a_i$, hence $f\downarrow$ is stable. 
 We can consider the corresponding conditional function $\textbf{f}\downarrow$.

Thereby, we define the functor $G(x):=\textbf{x}\downarrow$, $G(f):=\textbf{f}\downarrow$. Let us construct the inverse functor. 
Suppose now that $\textbf{X}$ is a conditional set of $X$ and $\A$. We will construct from $X$ an element $\tilde{X}$ of $\VA$. Indeed, for every $u\in X$ we define $\tilde{u}:\mathcal{D}_u\rightarrow\A$ where $\mathcal{D}_u:=\left\{\check{v}\colon v\in X \right\}$, and $\tilde{u}(\check{v})=a_{u,v}$ with $a_{u,v}:=\bigvee\left\{ b\in\A \colon u|a=v|a \right\}$ for each $v\in X$. Notice that $u|a_{u,v} = v|a_{u,v}$. The proof is similar to others we have done before: take a maximal disjoint family of elements $b$ such that $u|b = v|b$ and then use uniqueness of (C3) of Definition \ref{def: condSet}.

 Let $\tilde{X}:\mathcal{D}\rightarrow\A$ where
 $$\mathcal{D}=\left\{\bar{u}\colon u\in X \right\}\textnormal{ and } \tilde{X}(\bar{u})=1\textnormal{ for each }u\in X.$$ 

One has that $\tilde{X}$ is an element of $\VA$. Moreover, we claim that $\llbracket \bar{u}=\bar{v}\rrbracket=a_{u,v}$ for all $u,v\in X$.
 Indeed, 
\begin{equation}
\label{eqII}
\llbracket \bar{u}=\bar{v}\rrbracket=\underset{t\in X}\bigwedge \left(a_{u,t}\Rightarrow\llbracket \check{t}\in\bar{v}\rrbracket \right)\wedge\underset{s\in X}\bigwedge \left(a_{v,s}\Rightarrow\llbracket \check{s}\in\bar{u}\rrbracket \right).
\end{equation}

In addition,
\[
\llbracket \check{t}\in\bar{v}\rrbracket=\underset{w\in X}\bigvee a_{v,w}\wedge\llbracket \check{t}=\check{w}\rrbracket=a_{v,t},
\]
because $\llbracket \check{t}=\check{w}\rrbracket=1$ if $t=w$, $\llbracket \check{t}=\check{w}\rrbracket=0$ otherwise.
 Similarly, one has $\llbracket \check{s}\in\bar{u}\rrbracket=a_{u,s}$.

Therefore, replacing in (\ref{eqII}), one has 
\[
\llbracket \bar{u}=\bar{v}\rrbracket=\underset{t\in X}\bigwedge (a_{u,t}^c\vee a_{v,t}) \wedge\underset{s\in X}\bigwedge (a_{v,s}^c\vee a_{u,s}) .
\]

By considering above $t=u$ and $s=v$, we obtain
\[
\llbracket \bar{u}=\bar{v}\rrbracket\leq a_{u,v}
\]

For the converse inequality, suppose  by contradiction that $0< a:=a_{u,v}\wedge(a_{u,t}^c\vee a_{v,t})^c$ for some $t$. Since $a\leq a_{u,v},a_{u,t}$ one has $v|a=u|a=t|a$ by (C2).  But $a\leq a_{v,t}^c$ implies that $v|a\neq t|a$, which is a contradiction. 

For any $u\in X$, let $\tilde{u}$ denote the canonical representative of $\bar{u}$ in $\overline{V}^{(\A)}$. 
We claim that the map $X\rightarrow\tilde{X}\downarrow$ given by $u\rightarrow\tilde{u}$ is one-to-one. 
Indeed, if $\tilde{u}=\tilde{v}$, then $1=\llbracket\bar{u}=\bar{v} \rrbracket=a_{u,v}$, hence $u=v$. On the other hand, given $w\in \tilde{X}\downarrow$, one has 
$$1=\llbracket w\in \tilde{X}\rrbracket=\underset{u\in X}\bigvee\llbracket\bar{u}=w\rrbracket.$$ 
As we have done before, we can find by maximality a partition $\{a_i\}\in p(1)$ so that $a_i\leq\llbracket \bar{u}_i=w\rrbracket$ for some $u_i\in X$, each $i$. 
Then, (C3) of Definition \ref{def: condSet} provides us with $u\in X$ such that $u|a_i=u_i|a_i$ for all $i$.
We have that $\llbracket \bar{u}=\bar{u}_i\rrbracket=a_{u,u_i}\geq a_i$. Hence $a_i \leq \llbracket \bar{u}=\bar{u}_i\rrbracket \wedge \llbracket \bar{u}_i=w\rrbracket$ for all $i$, and so $\llbracket w=\bar{u}\rrbracket=1$ and thus $\tilde{u}=w$. 

Now suppose that $\textbf{f}:\textbf{X}\rightarrow \textbf{Y}$ is a conditional function between the conditional sets $\textbf{X},\textbf{Y}$.
 We consider the stable function $f:X\rightarrow Y$. 
 Let $g:\tilde{X}\downarrow\rightarrow \tilde{Y}\downarrow$ be with $g(\tilde{x}):=(f(x))^{\sim}$, which is well defined since the map $x\mapsto\tilde{x}$ is one-to-one. 
 Given $x,y\in X$, using that $f$ is stable we can show that  $\llbracket \tilde{x}=\tilde{y}\rrbracket=a_{x,y}\leq a_{f(x),f(y)}=\llbracket g(\tilde{x})=g(\tilde{y})\rrbracket$. Due to Theorem \ref{thm: extensional}, we can find $\tilde{f}$ in $\VA$ with $\llbracket \tilde{f}:\tilde{X}\rightarrow\tilde{Y}\rrbracket=1$ and such that $\llbracket g(\tilde{x})=\tilde{f}(x) \rrbracket=1$ for all $x\in X$.

Thereby, we take the functor $H(\textbf{X}):=\tilde{X}$ and $H(\textbf{f}):=\tilde{f}$. 
We will show that $G$ and $H$ are inverse equivalences. 
Suppose that $x$ is an element of $\VA$ with $\llbracket x\neq \emptyset\rrbracket=1$. 
We consider the map $x\downarrow\rightarrow ((x \downarrow)^{\sim})\downarrow$, $u\mapsto\tilde{u}$. 
Due to Theorem \ref{thm: extensional} it defines a name for a bijection between $x$ and $(x\downarrow)^{\sim}$. 
It follows by inspection that there is a natural isomorphism between $HG$ and the identity functor.

If $\textbf{X}$ is a conditional set, then we can consider the mapping $X\mapsto (\tilde{X})\downarrow$, $x\mapsto\tilde{x}$. This is a stable bijection, which defines a conditional bijection between the conditional sets $\textbf{X}$ and $\tilde{\textbf{X}}\downarrow$.    
 This also gives a natural isomorphism between $GH$ and the identity functor.
\end{proof}

\begin{rem}
The Boolean-valued part of the proof of Theorem \ref{thm: connectIII} is  covered by the well-known theorem from Boolean-valued analysis stating the equivalence of the category of names for non-empty sets and names for functions and the category of non-empty mix-complete Boolean sets and contractive functions  (see Kusraev and Kutateladze~\cite[Theorem 3.5.10]{kusraev2012boolean}). 
Thus, Theorem \ref{thm: connectIII} actually establishes that the category of conditional sets of $\A$ and conditional functions is equivalent to the category of non-empty mix-complete Boolean sets over $\A$ and contractive functions.
\end{rem}

One more time, the important message is not the equivalence of categories provided above, but that for any conditional set $\textbf{X}$ we build a tailored name $\tilde{X}$ of a set that induces a conditional set $\tilde{\textbf{X}}\downarrow$ which is essentially $\textbf{X}$.

Let us fix a conditional set $\textbf{X}$. 
For the forthcoming discussion, we will suppose w.l.o.g. that $\textbf{X}=\tilde{\textbf{X}}\downarrow$.  

%
%
%

Next, we will briefly explain how the main elements of the framework of conditional sets are connected to Boolean-valued analysis. 
A comprehensive introduction to conditional set theory is given in \cite{DJKK13}, thus for each unexplained notion we will give an exact reference to its definition in \cite{DJKK13}:

\begin{itemize}

\item \emph{Conditional subsets}: A non-empty subset $S$ of $X$ is \emph{stable} if $\sum x_i|a_i\in S$ whenever $\{x_i\}\subset S$ and $\{a_i\}\in p(1)$. 
A \emph{conditional subset} of $\textbf{X}$ is a conditional set $\textbf{S}:=\{x|a \colon x\in S,\: a\in\A \}$, where $S$ is a stable subset of $X$. For short, we will write $\textbf{S}\sqsubset\textbf{X}$.

Suppose that $\textbf{S}\sqsubset\textbf{X}$.  
We define $\hat{S}:\mathcal{D}_\textbf{S}\rightarrow\A$ with $\mathcal{D}_{\textbf{S}}:=\left\{\tilde{x}\colon x\in S\right\}$ and $\hat{S}(\tilde{x}):=1$. 
Then it can be verified that $\hat{S}$ is a name with $\llbracket\hat{S}\subset\tilde{X}\rrbracket=1$ and $\hat{\textbf{S}}\downarrow=\textbf{S}$. 

Now, suppose that $S_0$ is a name with $\llbracket\emptyset\neq S_0\subset\tilde{X}\rrbracket=1$. 
Then $\textbf{S}_0\downarrow\sqsubset\textbf{X}$ and $\llbracket S_0=(S_0\downarrow)^\wedge\rrbracket=1$.    

\item \emph{Conditional power set}: Let $P(\textbf{X})$ be the collection of all stable subsets of $E$. 
For $S\in P(\textbf{X})$ and $a\in\A$, we define $\textbf{S}|a:=\{x|b \colon x\in S,\: b\leq a\}$. 
The set $\textbf{P}(\textbf{X}):=\{\textbf{S}|a \colon S\textnormal{ is stable, }a\in\A\}$ is a conditional set which is called \emph{conditional power set}.

Suppose that $\textbf{C}\sqsubset\textbf{P}(\textbf{X})$.  
Let  $\hat{C}:\mathcal{D}_\textbf{C}\rightarrow\A$ with $\mathcal{D}_{\textbf{C}}:=\left\{\hat{S}\colon S\in C\right\}$ and $\hat{C}(\hat{S}):=1$.  
Then $\hat{C}$ is a name for a set of subsets of $\tilde{X}$. 

Now, given a name $C_0$ for a non-empty collection of non-empty sets of $\tilde{X}$, we define $C_0\Downarrow:=\{S\downarrow  \colon S\in C_0\downarrow \}$. 
Then $C_0\Downarrow$ is a stable set of subsets of $X$ and we can consider the corresponding conditional set $\textbf{C}_0\Downarrow\sqsubset \textbf{P}(\textbf{X})$. 

Moreover, if $\textbf{C}\sqsubset\textbf{P}(\textbf{X})$  one has that $\hat{\textbf{C}}\Downarrow=\textbf{C}$ and if  $C_0$ is a name for a non-empty collection of non-empty sets of $\tilde{X}$ one has $\llbracket C_0=(C_0\Downarrow)^\wedge   \rrbracket=1$.

In particular, if $\textbf{C}=\textbf{P}(\textbf{X})$, then $\hat{C}$ is a name for the collection of all non-empty subsets  of $\tilde{X}$ in $V^{(\A)}$. 

\item \emph{Conditional step functions}: If $E$ is a non-empty set, consider the \emph{conditional set of step functions}, let us say $\textbf{E}_\textbf{s}$, see \cite[Examples 2.3(5)]{DJKK13}. Then, the name $\tilde{E}_s$ is precisely the canonical name $\check{E}$ of $E$ in $V^{(\A)}$.

The \emph{conditional natural numbers} $\textbf{N}$ and the \emph{conditional rational numbers} $\textbf{Q}$ are  introduced in \cite{DJKK13} as a particular case of the step functions. 
It is known that $\llbracket\N^{(\A)}=\check{\N}\rrbracket=\Omega$ and $\llbracket\Q^{(\A)}=\check{\Q} \rrbracket=\Omega$, see eg~\cite{takeuti2015two}. 
Thus, it is satisfied that $\tilde{N}$ and $\tilde{Q}$ are names for the natural numbers and the rational numbers of $V^{(\A)}$, respectively.
\item \emph{Conditional real numbers}: In \cite{DJKK13} a conditional set $\textbf{R}$ which is called  \emph{conditional real numbers} is defined, see \cite[Definition 4.3]{DJKK13}. 
Then it can be verified that $\tilde{R}$ is a name for the real numbers of $V^{(\A)}$. 
 
\item \emph{Conditional topologies}: Suppose that $\boldsymbol{\mathcal{T}}$ is a \emph{conditional topology} on $\textbf{X}$, see \cite[Definition 3.1]{DJKK13}. 
Then $\hat{\mathcal{T}}$ is a name for the set of non-empty open sets of a topology on $\tilde{X}$. 

If $\mathcal{T}_0$ is a name for  the set of non-empty open sets of a topology on $\tilde{X}$ then $\boldsymbol{\mathcal{T}}_0\Downarrow$ is a conditional topology. 

Moreover, 
$\textbf{O}$ is a \emph{conditional open} subset  if and only if $\hat{O}$ is a name for an open set. 
$\textbf{C}$ is a \emph{conditional closed} subset if and only if $\hat{C}$ is a name for a closed set. 
$\textbf{S}$ is a \emph{conditionally compact subset}  (see \cite[Definition 3.24]{DJKK13}) if and only if $\hat{S}$ is a name for a compact subset. 

Furthermore, $\boldsymbol{\mathcal{T}}$ is \emph{conditionally Hausdorff} (see \cite[Section 3]{DJKK13}) if and only if $\llbracket\textnormal{``}\mathcal{T}_0\textnormal{ is Hausdorff''} \rrbracket=1$.

\item \emph{Conditional functions:} Given a conditional function $\textbf{f}:\textbf{S}_1\rightarrow\textbf{S}_2$, where $\textbf{S}_1,\textbf{S}_2$ are conditional subsets of $\textbf{X}$, then we have a stable function $g:S_1\rightarrow S_2$. Theorem \ref{thm: extensional} allows to define a name $\hat{f}$ of a function from $S_1$ to $S_2$ with $\hat{f}\downarrow=f$. 

Conversely, if $f$ is name for a function between subsets of $\tilde{X}$, then $f\downarrow$ is a stable function between stable subsets of $X$ and it defines a conditional function $\textbf{f}\downarrow$  between conditional subsets of $\textbf{X}$.

The same applies to \emph{conditional families}, \emph{conditional nets} and \emph{conditional sequences}, see \cite[Definition 2.20]{DJKK13}.

\end{itemize}

Bearing in mind the construction given in the proof of Theorem \ref{thm: connectIII}, the following is easy to check: $\textbf{X}$ is a \emph{conditional metric space}, see \cite[Definition 4.5]{DJKK13}, if and only if $\tilde{X}$ is a name for a metric space;
 $\textbf{X}$ is a \emph{conditional locally convex space}, see \cite[Definition 5.4]{DJKK13}, if and only if $\tilde{X}$ is a name for a locally convex space;
 $\textbf{X}$ is a \emph{conditional normed space}, see \cite[Definition 5.11]{DJKK13}, if and only if $\tilde{X}$ is a name for a normed space;  $\textbf{X}$ is a \emph{conditional Banach space}, see \cite[Section 5]{DJKK13}, if and only if $\tilde{X}$ is a name for a Banach space.

Again, we see that all these objects are some of the building blocks for the main results provided in \cite{DJKK13}. 
Clearly, names for more and more conditional versions of classical objects can be defined by using the same logic. 

As an instance of application, we can provide a conditional version of the Schauder-Tychonov fixed point theorem:

\begin{prop}
Let $\textbf{X}$ be a conditional locally convex space which is conditionally Hausdorff. 
If $\textbf{C}$ is a conditionally compact conditional subset of $\textbf{X}$ and $\textbf{f}:\textbf{C}\rightarrow \textbf{C}$ is a conditionally continuous conditional function, then there exists $\textbf{x}$ in $\textbf{C}$ such that $\textbf{f}(\textbf{x})=\textbf{x}$.  
\end{prop}

We can consider the names $\tilde{X}$, $\hat{C}$ and $\hat{f}$ as described above.  
If $T$ denotes the statement of the Schauder-Tychonov Theorem, then the statement above, let us say $\textbf{T}$, is nothing else but a reformulation of the statement $\Vert T\Vert=1$ with the Boolean truth value of $\VA$. 
By the transfer principle of Boolean-valued models, one has that $\Vert T\Vert=1$ holds, thus  $\textbf{T}$ is also a theorem.
 
Of course, this is just an example. 
In general, this method can be systematically applied to the different theorems of \cite{DJKK13}.\\


\begin{thebibliography}{10}

\bibitem{artzner01}
P.~Artzner, F.~Delbaen, J.~M. Eber, and D.~Heath.
\newblock Coherent measures of risk.
\newblock {\em Mathematical Finance}, 9:203--228, 1999.

\bibitem{BH14}
J.~Backhoff and U.~Horst.
\newblock {Conditional analysis and a Principal-Agent problem}.
\newblock {\em SIAM Journal on Financial Mathematics}, 7(1):477–507, 2016.

\bibitem{barbu2012convexity}
V.~Barbu and T.~Precupanu.
\newblock {\em Convexity and optimization in Banach spaces}.
\newblock Springer Science \& Business Media, 2012.

\bibitem{bell2005set}
J.L. Bell.
\newblock {\em Set Theory: Boolean-Valued Models and Independence Proofs}.
\newblock Oxford Logic Guides. Clarendon Press, 2005.

\bibitem{bielecki2016dynamic}
T.R. Bielecki, I.~Cialenco, S.~Drapeau, and M.~Karliczek.
\newblock Dynamic assessment indices.
\newblock {\em Stochastics}, 88(1):1--44, 2016.

\bibitem{borwein2006variational}
J.~M. Borwein and Q.~J. Zhu.
\newblock Variational methods in convex analysis.
\newblock {\em Journal of Global Optimization}, 35(2):197--213, 2006.

\bibitem{kupper08}
P.~Cheridito, U.~Horst, M.~Kupper, and T.~Pirvu.
\newblock {Equilibrium pricing in incomplete markets under translation
  invariant preferences}.
\newblock {\em Mathematics of Operations Research}, 41(1):174 -- 195, 2016.

\bibitem{Cheridito2012}
P.~Cheridito, M.~Kupper, and N.~Vogelpoth.
\newblock {Conditional analysis on $\mathbb{R}^d$}.
\newblock {\em Set Optimization and Applications, Proceedings in Mathematics \&
  Statistics}, 151:179 -- 211, 2015.

\bibitem{cohen1966set}
P.~J. Cohen.
\newblock Set theory and the continuum hypothesis. w.a. benjamin.
\newblock {\em Inc., New York}, 1966.

\bibitem{DJ13}
S.~Drapeau and A.~Jamneshan.
\newblock {Conditional preferences and their numerical representations}.
\newblock {\em Journal of Mathematical Economics}, 63:106--118, 2016.

\bibitem{DJKK13}
S.~Drapeau, A.~Jamneshan, M.~Karliczek, and M.~Kupper.
\newblock {The algebra of conditional sets, and the concepts of conditional
  topology and compactness}.
\newblock {\em Journal of Mathematical Analysis and Applications}, 437(1):561--
  589, 2016.

\bibitem{DKKM13}
S.~Drapeau, M.~Karliczek, M.~Kupper, and M.~Streckfuss.
\newblock {Brouwer fixed point theorem in $(L^0)^d$}.
\newblock {\em Fixed Point Theory and Applications}, 301(1), 2013.

\bibitem{eisele13}
K-T. Eisele and S.~Taieb.
\newblock Weak topologies for modules over rings of bounded random variables.
\newblock {\em Journal of Mathematical Analysis and Applications},
  421(2):1334--1357, 2015.

\bibitem{kupper03}
D.~Filipovi{\'c}, M.~Kupper, and N.~Vogelpoth.
\newblock {Separation and duality in locally $L^0$-convex modules}.
\newblock {\em Journal of Functional Analysis}, 256:3996 -- 4029, 2009.

\bibitem{kupper11}
D.~Filipovi{\'c}, M.~Kupper, and N.~Vogelpoth.
\newblock {Approaches to conditional risk}.
\newblock {\em SIAM Journal of Financial Mathematics}, 3(1):402 -- 432, 2012.

\bibitem{follmer2002convex}
H.~F{\"o}llmer and A.~Schied.
\newblock Convex measures of risk and trading constraints.
\newblock {\em Finance and stochastics}, 6(4):429--447, 2002.

\bibitem{frittelli2002putting}
M.~Frittelli and E.~R. Gianin.
\newblock Putting order in risk measures.
\newblock {\em Journal of Banking \& Finance}, 26(7):1473--1486, 2002.

\bibitem{frittelli2011dual}
M.~Frittelli and M.~Maggis.
\newblock Dual representation of quasi-convex conditional maps.
\newblock {\em SIAM Journal on Financial Mathematics}, 2(1):357--382, 2011.

\bibitem{gordon}
E.~I. Gordon.
\newblock {$K$-spaces in Boolean-valued models of set theory}.
\newblock {\em Dokl. Akad. Nauk SSSR}, 258(4):777--780, 1981.

\bibitem{gordon1983rationally}
E.~I. Gordon.
\newblock Rationally complete semiprime commutative rings in boolean valued
  models of set theory.
\newblock {\em Gor ki{\i}, VINITI}, (3286-83), 1983.

\bibitem{guo08}
T.~Guo.
\newblock {The relation of Banach-Alaoglu theorem and
  Banach-Bourbaki-Kakutani-\v{S}mulian theorem in complete random normed
  modules to stratification structure}.
\newblock {\em Science in China Series A Mathematics}, 51:1651--1663, 2008.

\bibitem{guo10}
T.~Guo.
\newblock {Relations between some basic results derived from two kinds of
  topologies for a random locally convex module}.
\newblock {\em Journal of Functional Analysis}, 258:3024--3047, 2010.

\bibitem{guo2013homo}
T.~Guo.
\newblock {On Some Basic Theorems of Continuous Module Homomorphisms between
  Random Normed Modules}.
\newblock {\em Journal of Function Spaces and Applications}, pages 1--13, 2013.

\bibitem{guo2009random}
T.~Guo and X.~Chen.
\newblock Random duality.
\newblock {\em Science in China Series A: Mathematics}, 52(10):2084--2098,
  2009.

\bibitem{guo2015random1}
T.~Guo, S.~Zhao, and X.~Zeng.
\newblock {Random convex analysis (I): separation and Fenchel-Moreau duality in
  random locally convex modules}.
\newblock {\em arXiv preprint arXiv:1503.08695}, 2015.

\bibitem{gutman2009boundedness}
A.~E. Gutman and S.~A. Lisovskaya.
\newblock The boundedness principle for lattice-normed spaces.
\newblock {\em Siberian Mathematical Journal}, 50(5):830--837, 2009.

\bibitem{haydon1991randomly}
R.~Haydon, M.~Levy, and Y.~Raynaud.
\newblock {\em Randomly normed spaces}.
\newblock Hermann, 1991.

\bibitem{JKZ2017control}
A.~Jamneshan, M.~Kupper, and J.~M. Zapata.
\newblock {Parameter-dependent stochastic optimal control in finite discrete
  time}.
\newblock {\em arXiv preprint arXiv:1705.02374}, 2017.

\bibitem{L0compactness}
A.~Jamneshan and J.~M. Zapata.
\newblock {On compactness in $L^0$-modules}.
\newblock {\em arXiv preprint arXiv:1711.09785}, 2017.

\bibitem{jech2013set}
T.~Jech.
\newblock {\em Set theory}.
\newblock Springer Science \& Business Media, 2013.

\bibitem{jouini2006law}
E.~Jouini, W.~Schachermayer, and N.~Touzi.
\newblock Law invariant risk measures have the fatou property.
\newblock {\em Advances in mathematical economics}, pages 49--71, 2006.

\bibitem{Kantorovich}
L.~V. Kantorovich.
\newblock {\em To the general theory of operations in semiordered spaces},
  volume~1.
\newblock Dokl. Akad. Nauk SSSR (Russian), 1936.

\bibitem{cyclicCompactness}
A.~G. Kusraev.
\newblock Boolean valued analysis of duality between universally complete
  modules.
\newblock {\em Dokl. Akad. Nauk SSSR}, 267(5):1049--1052, 1982.

\bibitem{kusraev1985banach}
A.~G. Kusraev.
\newblock Banach-kantorovich spaces.
\newblock {\em Siberian Mathematical Journal}, 26(2):254--259, 1985.

\bibitem{kusraev1985vector}
A.~G. Kusraev.
\newblock Vector duality and its applications, 1985.

\bibitem{kusraev2000dominated}
A.~G. Kusraev.
\newblock Dominated operators.
\newblock In {\em Dominated Operators}, pages 141--186. Springer, 2000.

\bibitem{kusraev2012subdifferentials}
A.~G. Kusraev and S.~S. Kutateladze.
\newblock {\em Subdifferentials: Theory and applications}, volume 323.
\newblock Springer Science \& Business Media, 2012.

\bibitem{kusraev2014boolean}
A.~G. Kusraev and S.~S. Kutateladze.
\newblock Boolean valued analysis: Selected topics.
\newblock {\em Vladikavkaz: SMI VSC RAS}, 1000(6), 2014.

\bibitem{kusraev2012boolean}
A.~G. Kusraev and S.~S. Kutateladze.
\newblock {\em Boolean Valued Analysis}.
\newblock Mathematics and Its Applications. Springer Netherlands, 2012.

\bibitem{orihuela2012coercive}
J.~Orihuela and M.~Ruiz-Gal{\'a}n.
\newblock A coercive james’s weak compactness theorem and nonlinear
  variational problems.
\newblock {\em Nonlinear Analysis: Theory, Methods \& Applications},
  75(2):598--611, 2012.

\bibitem{OZ2017stabil}
J.~Orihuela and J.~M. Zapata.
\newblock {Stability in locally $L^0$-convex modules and a conditional version
  of James' compactness theorem}.
\newblock {\em Journal of Mathematical Analysis and Applications}, 452(2):1101
  -- 1127, 2017.

\bibitem{saint2013weak}
J.~Saint-Raymond.
\newblock Weak compactness and variational characterization of the convexity.
\newblock {\em Mediterranean journal of mathematics}, 10(2):927--940, 2013.

\bibitem{scott1967proof}
D.~Scott.
\newblock A proof of the independence of the continuum hypothesis.
\newblock {\em Theory of Computing Systems}, 1(2):89--111, 1967.

\bibitem{takeuti2015two}
G.~Takeuti.
\newblock {\em Two Applications of Logic to Mathematics}.
\newblock Publications of the Mathematical Society of Japan. Princeton
  University Press, 2015.

\bibitem{vladimirov2013boolean}
D.~A. Vladimirov.
\newblock {\em Boolean algebras in analysis}, volume 540.
\newblock Springer Science \& Business Media, 2013.

\bibitem{zapataportfolio}
J.M. Zapata.
\newblock {Randomized versions of Mazur lemma and Krein-\v{S}mulian theorem}.
\newblock {\em Journal of Convex Analysis}, 25(3), 2018.

\bibitem{zapataweak}
J.M. Zapata.
\newblock {Versions of Eberlein-{\v{S}}mulian and Amir-Lindenstrauss theorems
  in the framework of conditional sets}.
\newblock {\em Applicable Analysis and Discrete Mathematics}, 10(2):231--261,
  2016.

\bibitem{zapata2017characterization}
J.M. Zapata.
\newblock {On the Characterization of Locally $L^0$-Convex Topologies Induced
  by a Family of $L^0$-Seminorms}.
\newblock {\em Journal of Convex Analysis}, 24(2):383--391, 2017.

\end{thebibliography}
\end{document}